\documentclass[12pt]{amsart}

\usepackage{a4wide,amsmath,amssymb,enumerate}
\usepackage{graphicx}
\usepackage[usenames,dvipsnames,svgnames,table]{xcolor}
\usepackage[colorlinks=true, pdfstartview=FitV, linkcolor=ForestGreen,citecolor=ForestGreen, urlcolor=blue]{hyperref}

\newtheorem{thm}{Theorem}[section]
\newtheorem{cor}[thm]{Corollary}
\newtheorem{lem}[thm]{Lemma}
\newtheorem{prop}[thm]{Proposition}

\theoremstyle{definition}

\theoremstyle{definition}

\theoremstyle{definition}
\newtheorem{ex}[thm]{Example}
\theoremstyle{definition}
\newtheorem{defn}[thm]{Definition}

\theoremstyle{remark}
\newtheorem{rem}[thm]{Remark} 

\numberwithin{equation}{section}

\newcommand{\R}{\mathbb R}

\newcommand{\Sd}{\mathbb{S}_{d+1}(\R)}

\def\1{\mathbb I}

\def\e{\varepsilon}

\def\eps{\varepsilon}

\DeclareMathOperator{\trace}{Trace}

\begin{document}

\title[Effective junction conditions for parabolic equations]{Effective
  junction conditions for degenerate parabolic equations}

\author{Cyril Imbert}
\address[C.~Imbert]{CNRS \& Department of Mathematics and Applications, \'Ecole Normale Sup\'erieure (Paris) \\
45 rue d'Ulm, 75005 Paris, France}
\email{Cyril.Imbert@ens.fr}
\author{Vinh Duc Nguyen}
\address[V.~Nguyen]{CERMICS, \'Ecole Nationale des Ponts et Chauss\'ees, Universit\'e Paris-Est \\
 6 et 8 avenue Blaise Pascal, 77455 Marne-La-Vallée Cedex 2, France}
\email{Vinh.Nguyen@math.cnrs.fr}

\date{\today}

\begin{abstract}
  We are interested in the study of parabolic equations on \emph{a
    multi-dimensional junction}, i.e. the
  union of a finite number of copies of a half-hyperplane of dimension
  $d+1$ whose boundaries are identified. The common boundary is
  referred to as the \emph{junction hyperplane}.  The parabolic
  equations on the half-hyperplanes are in non-divergence form, fully
  non-linear and possibly degenerate, and they do degenerate and are
  quasi-convex along the junction hyperplane. More precisely, along
  the junction hyperplane the nonlinearities do not depend on second
  order derivatives and their sublevel sets with respect to the
  gradient variable are convex. The parabolic equations are
  supplemented with a non-linear boundary condition of Neumann type,
  referred to as a \emph{generalized junction condition}, which is
  compatible with the maximum principle. Our main result asserts that
  imposing a generalized junction condition in a weak sense reduces to
  imposing an \emph{effective one} in a strong sense.  This result
  extends the one obtained by Imbert and Monneau for Hamilton-Jacobi
  equations on networks and multi-dimensional junctions. We give two
  applications of this result.  On the one hand, we give the first
  complete answer to an open question about these equations: we prove
  in the two-domain case that the vanishing viscosity limit associated
  with quasi-convex Hamilton-Jacobi equations coincides with the
  maximal Ishii solution identified by Barles, Briani and Chasseigne
  (2012). On the other hand, we give a short and simple PDE proof of a
  large deviation result of Bou\'e, Dupuis and Ellis (2000).
\end{abstract}



\maketitle

\setcounter{tocdepth}{1}
\tableofcontents

\section{Introduction}

\subsection{Degenerate parabolic equations on junctions}

\emph{Multi-dimensional junctions} \cite{fw04,im14,oudet} are union of
half-spaces whose boundaries are identified -- see
Figure~\ref{fig:junction}. Precisely:
\[
J=\bigcup_{i=1}^N J_i \quad 
\text{ with } \quad 
\begin{cases} J_i = \{ x= (x',x_i): x' \in \R^d, x_i
  \ge 0 \} \simeq \R^{d+1}_+ \smallskip\\ 
J_i\cap J_j = \Gamma \simeq \R^d  \quad \text{ for }\quad i\not=j. 
\end{cases}
\]
\begin{figure}
\includegraphics[height=4cm]{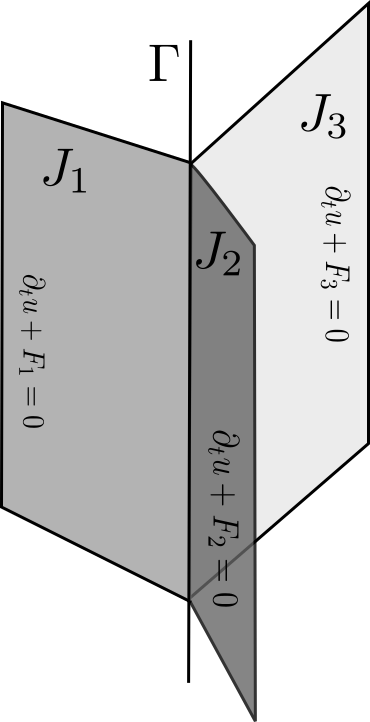}
\caption{A parabolic equation posed on a multi-dimensional
  junction. Here there are 3 branches ($N=3$) and the tangential
  dimension is $1$ ($d=1$). We did not illustrate the junction
  condition $L=0$ on the junction hyperplane $\Gamma$ (which is a line
  in this example).}
\label{fig:junction}
\end{figure}
Given $T \in [0,+\infty]$, we consider a general degenerate parabolic
equation posed on a junction,
\begin{eqnarray}
\label{eq:dp-gen}
&& \left\{
\begin{array}{ll}
u_t +  F_i(t,x,Du,D^2u)=0 & (t,x)\in (0,T) \times J_i^*, i=1,\dots, N, \\
L(-u_t,\partial_1 u,\dots,\partial_N u,t,x',D'u)=0 & (t,x) \in (0,T)\times \Gamma
\end{array}
\right.
\end{eqnarray}
where $J_i^*$ denotes $J_i \setminus \Gamma$, $u_t$ denotes the time
derivative, $Du$ and $D^2u$ respectively denote the gradient and the
Hessian of $u$ with respect to $x$, and for $x' \in \Gamma$,
$\partial_i u(x')$ denotes the derivative of $u_i(x)= u|_{J_i}(x)$
with respect to $x_i$ at $x_i=0$ (recall $x=(x',x_i)$) and $D'u$
denotes the derivative with respect to $x'$. 
\begin{ex}
  The case $N=1$ corresponds to the study of a degenerate parabolic
  equation posed on a half-space, subject to a non-linear boundary
  condition (dynamic or not). Example~\ref{ex:1d} illustrates how the
  main theorem can be applied in this special case.  The case $N=2$
  corresponds to the \emph{two-domain case}: a degenerate parabolic
  equation has coefficients which are continuous on either part of a
  hyperplane (or a smooth interface); the generalized junction
  condition can be thought as a transmission
  condition. Theorem~\ref{thm:vvl-evol-tilde} is an application of the
  main theorem with  $N=2$.
\end{ex}

We make the following assumptions on
each $F_i$. 
\medskip

\noindent \textbf{Assumption (F).}
\begin{itemize}
\item[\bf (F1).]  The function $F_i$  is continuous and degenerate elliptic.
\item[\bf (F2).] For all  $R>0$,  there exists  $C_{i,R} >0$   such that for all 
 $y=(y',y_i)$, all $p \in \R^{d+1}$, all $B \in \Sd$ and all $\lambda >0$
\[  
\left. \begin{array}{r} 
s \in (0,T), |y_i|\le 1 \\ |y'| +|B| \le R 
\end{array}
\right\} \Rightarrow
 F_i(s,y,p,B + \lambda e_{d+1} \otimes e_{d+1}) \ge F_i (s,y,p,B) 
- C_{i,R}\lambda |y_i|^2.
\]
\item[\bf (F3).] For all  $R>0$,
\[ \lim_{|p|\to +\infty} \inf_{t \in (0,T), |x| + |B| \le R} F_i(t,x,p,B) = +\infty. \]
\item[\bf (F4).] There exists
  $H_i : (0,T) \times \Gamma \times \R^d \times \R \to \R$ continuous
  such that
\begin{itemize}
\item for all
  $(t,x',p',p_i,B) \in (0,T) \times \Gamma \times \R^{d}\times \R
  \times \Sd$, $F_i(t,(x',0),(p',p_i),B) = H_i (t,x',p',p_i)$;
\item for all $t\in (0,T)$, $x' \in \Gamma$, for all $\lambda \in \R$,
  the set $\{ p=(p',p_i) \in \R^{d+1} : H_i (t,x',p) \le \lambda \}$ is
  convex.
\end{itemize}
\end{itemize}
In the assumption above, $\Sd$ denotes the set of real-valued
$(d+1)\times (d+1)$ symmetric matrices and $e_{d+1}$ denotes the unit
vector orthogonal to $\Gamma$ and pointing inside $J_i$.  We recall
that $F_i(t,x,p,A)$ is degenerate elliptic if it is non-increasing with
respect to $A$ (using the classical partial order on $\Sd$).  The
function $H_i$ appearing in (F) is referred to as the
\emph{Hamiltonian} from the branch $J_i$.
\begin{ex}[First order case]\label{ex:one}
  The first example we give is the one coming from \cite{im13,im14}.
  It reduces to deal with $F_i(t,x,p,B) = H_i(t,x,p)$ for any $x \in J_i$
  (and not only $x=(x',0) \in \Gamma$) and $p \in \R^d$ with $H_i$
  continuous, coercive in $p$ uniformly in $x$, i.e. satisfying
  \[ \lim_{|p|\to +\infty} \inf_{(t,x) \in (0,T) \times J} H_i(t,x,p) = +\infty\]
  and quasi-convex in $p$, i.e. the sublevel sets
  $\{ p \in \R^{d+1}: H_i (t,x,p) \le \lambda \}$ are convex for all
  $\lambda \in \R$ and $(t,x) \in (0,T)\times \Gamma$.
\end{ex}
\begin{ex}[The model case]
  Our results apply to the model case where 
  \(F_i(t,x,p,B) = H_i (t,x,p) - \trace (\sigma_i(x) \sigma_i^T(x) B)\)
  with $H_i$ is as in Example~\ref{ex:one} and where the
  $(d+1)\times m$ real matrix $\sigma_i$ is such that
  $\sigma_i \equiv 0$ on $\Gamma$ and the $(d+1)$-th line
  $\sigma_i^{d+1}$ of $\sigma_i$ satisfies
  $|\sigma_i^{d+1} (y)|\le c_i |y_{d+1}|$. Remark that this latter
  condition holds true if $\sigma_i \equiv 0$ on $\Gamma$ and
  $\sigma_i$ is Lipschitz continuous.
\end{ex}
 As far as the junction function $L$ is concerned, we make the following 
assumption. 
\medskip

\noindent \textbf{Assumption (L).}
\begin{itemize}
\item[\bf (L1).]  The function $L$ is continuous.
\item[\bf (L2).] The function $L(p_0,\dots,p_N,t,x',p')$ is
  non-increasing in $p_i$ for $i=0,\dots,N$.
\item[\bf (L3).]  $\forall i, p_i < q_i  
  \Rightarrow L(p_0,\dots, p_N,t,x',p') > L(q_0,\dots,q_N,t,x',p')$.
\item[\bf (L4).]
  $ \displaystyle \inf_{t,x',p'} L(p_0,\dots,p_N,t,x',p') \to + \infty
  \text{ as } \min_{i=1,\dots,N} p_i \to -\infty.$
\item[\bf (L5).]
$\displaystyle \sup_{t,x',p'}  L(p_0,\dots,p_N,t,x',p') \to - \infty \text{ as } 
  \max_{i=0,\dots,N} p_i \to +\infty$.
\end{itemize}
\begin{ex}[Kirchoff conditions] \label{ex:kirchoff}
A model for $L$ is 
\[ L (p_0,\dots,p_N) = - \sum_{i=1}^N \beta_i p_i \]
with $\beta_i > 0$ for all $i$.  Such a condition is called a Kirchoff
condition.
\end{ex}
\begin{ex}[Flux-limited junction conditions]
A second important example of junction functions $L$ is the one related to
flux-limited solutions \cite{im13,im14}.  Given a \emph{flux limiter}
$A$,
\[ 
\begin{cases} 
  A : (0,T) \times \Gamma \times \R^d \to \R \text{ continuous } \\
  \text{ for all } (t,x') \in (0,T) \times \Gamma, \lambda \in \R, \{ p'
  \in \R^d: A (t,x',p') \le \lambda \} \text{ convex }
\end{cases}
\]
we consider the associated junction function $L_A$ defined by
\begin{equation} \label{eq:defi-la} 
L_A (p_0,\dots,p_N,t,x',p') = -p_0 + \max( A (t,x',p'), \max_i H_i^- (t,x',p',p_i))
\end{equation}
where $H_i^-(t,x,p',p_i)$ denotes the non-increasing part of
$p_i \mapsto H_i (t,x',p',p_i)$ \cite{im14}: if
$p_i \mapsto H_i (t,x',p',p_i)$ reaches its minimum at $\pi_i^0(t,x',p')$,
which is the minimal minimizer, then
\[ H_i^- (t,x',p',p_i) = \begin{cases} H_i (t,x',p',p_i) & \text{ if } p_i \le \pi_i^0(t,x',p') \\
H_i (t,x',p',\pi_i^0 (t,x',p')) & \text{ if } p_i \ge \pi_i^0(t,x',p').
\end{cases}\]
\end{ex}
\begin{rem}
  The flux-limited function $F_A$ defined in \cite{im13,im14}
  corresponds to 
  \begin{align*}
 F_A (p_1,\dots,p_N,t,x',p') &= \max( A (t,x',p'), \max_i H_i^-
  (t,x',p',p_i))\\
&= L_A (p_0,p_1,\dots,p_N,t,x',p')+p_0.
\end{align*}
\end{rem}
The appropriate notion of weak solutions for Hamilton-Jacobi equations
is the one of viscosity solutions, introduced by Crandall and Lions
\cite{cl83} -- see also \cite{cel}.  It is explained in
\cite{im13,im14} that two notions of viscosity solutions are needed in
the study of Hamilton-Jacobi equations on networks, depending on the
type of junction conditions we impose. We will see that it is also the
case for the degenerate parabolic equations we consider in this work.
For general junction functions $L$ in \eqref{eq:dp-gen}, the junction
condition has to be understood in the following weak sense: either the
junction condition $L=0$ or one of the equations $u_t + F_i=0$ is
satisfied. We refer to such viscosity solutions as \emph{relaxed
  solutions} -- see Definition~\ref{defi:relaxed} below.  But for the
special junction conditions $L_A$ given by \eqref{eq:defi-la}, relaxed
solutions satisfy the junction condition in a stronger sense: the
junction condition $L_A=0$ is indeed satisfied
(Proposition~\ref{prop:rel-fl}).  Such viscosity solutions are
referred to as \emph{flux-limited solutions} -- see
Definition~\ref{defi:FL} below.

\subsection{Main result}

The main result of this article is about \emph{equivalent classes of
  generalized junction conditions}. Roughly speaking, we prove that
imposing a general junction condition amounts to imposing an effective
one. This effective junction condition corresponds to some $L_A$ given
in \eqref{eq:defi-la} for some flux limiter $A=A_L$. This flux limiter
only depends on the junction function $L$ and the Hamiltonians
$H_i$.  Moreover, the \emph{effective junction condition} $L_{A_L}$ is
satisfied in a strong sense: if the relaxed solution $u$ is
continuously differentiable in time and space up to the junction
hyperplane $\Gamma$, then the boundary condition $L=0$ on $\Gamma$ can
be lost (see the discussion above and Definition~\ref{defi:relaxed})
but $L_{A_L}=0$ on $\Gamma$ is indeed satisfied in the classical
sense.
\begin{defn}[The effective flux limiter $A_L$] \label{defi:al} 
 Let 
\begin{equation}\label{a0}
 A_0 (t,x',p') = \max_{i=1,\dots,N} \min_{p_i \in \R} H_i(t,x',p',p_i)
\end{equation}
and $p_i^0 \ge \pi_i^0 (t,x',p')$ be the minimal $p_i$ such that
  \(H_i(t,x',p',p_i) =A_0(t,x',p').\)  
For all $(t,x',p')$, the \emph{effective flux limiter} 
  $A_L(t,x',p')$ is defined as follows: if
  \[L(A_0(t,x',p'),p_1^0,\dots,p_N^0,t,x',p') \le 0,\] then $A_L(t,x',p')=A_0(t,x',p')$,
else $A_L(t,x',p')$ is the only real number $\lambda \ge A_0 (t,x',p')$  such that
there exists $p_i^+ \ge p_i^0$ with
\[ H_i (t,x',p',p_i^+) = \lambda \quad \text{ and } \quad
L(\lambda,p_1^+,\dots,p_N^+,t,x',p')=0.\]
\end{defn}
\begin{rem}
  We will give in Section~\ref{sec:class} other representations of
  $A_L$ -- see Proposition~\ref{prop:al}.  We note that if $L$
  satisfies (L) then $\lambda$ is unique. But the $p_i^+$ are not (in
  general) -- see the case on the right at the top of
  Figure~\ref{fig:1} in Example~\ref{ex:1d} below.
\end{rem}
\begin{thm}[Effective junction conditions]
\label{thm:class}
Assume \emph{(F)}, \emph{(L)}. Then
$A_L: (0,T) \times \Gamma \times \R^d \to \R$ given in
Definition~\ref{defi:al} is well-defined, continuous, such that
\[ 
\lim_{|p'| \to + \infty} \inf_{(t,x') \in (0,T) \times \Gamma} A_L(t,x',p') = +\infty 
\] 
and such that any $L$-relaxed sub-solution (resp. super-solution) of
\eqref{eq:dp-gen} is an $A_L$-flux-limited sub-solution
(resp. super-solution) of \eqref{eq:dp-gen}. Moreover, if
\[ \{ p' \in \R^d: L(p_0,p_1,\dots,p_N,t,x',p') \le \lambda \} \text{
  is convex } \]
for all $p_0,\dots,p_N,\lambda \in \R$ and all
$(t,x') \in (0,T) \times \Gamma$, then
\[ \{ p' \in \R^d : A_L(t,x',p')\le \lambda \} \text{ is convex } \]
for all $\lambda \in \R$ and $(t,x') \in (0,T) \times \Gamma$.
\end{thm}
\begin{rem}
  Applying Theorem~\ref{thm:class} in the case $N=1$, \emph{effective
    boundary conditions} for degenerate parabolic equations posed on a
  domain are obtained; see Example~\ref{ex:1d} for instance.  In the
  case $N=2$, we get \emph{effective transmission conditions}; see
  Theorem~\ref{thm:vvl-evol-tilde} for instance. 
\end{rem}
\begin{figure}[t]
\includegraphics[height=4cm]{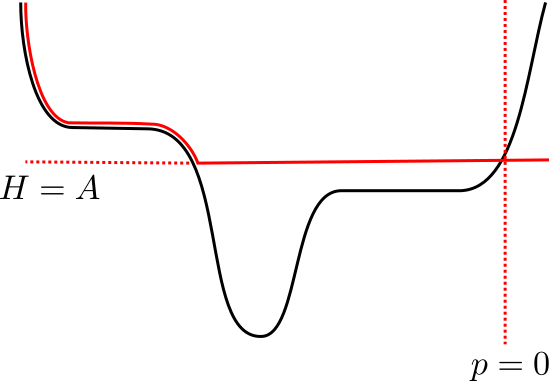}
\includegraphics[height=4cm]{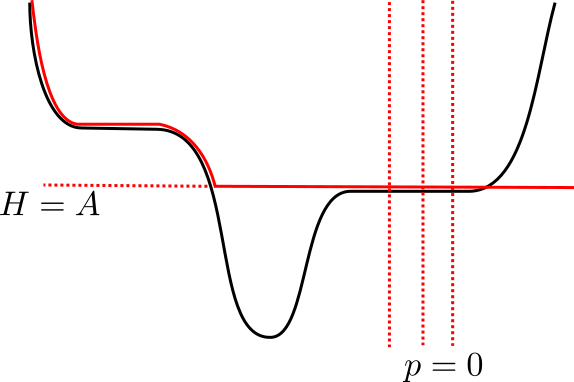}
\includegraphics[height=4cm]{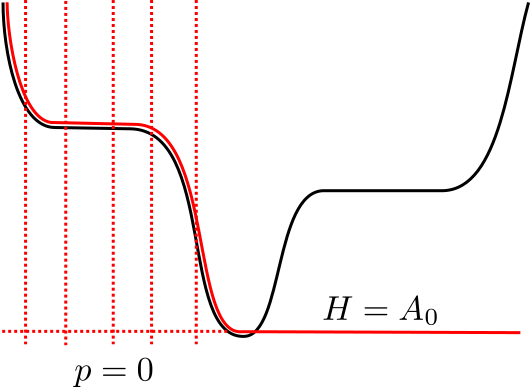}
\caption{This figure illustrates Example~\ref{ex:1d} where $N=1$
  (number of branches) and $d=0$ (dimension of the tangential
  space). The effective flux limiter $A$ is determined in each case by
  looking at the points where the vertical line $\{p=0\}$ intersects
  the graph of the Hamiltonian $H=H_1$; the variable $p$ reduces here
  to $p_1$ in the general setting. }
\label{fig:1}
\end{figure}
\begin{ex}[The 1D Neumann problem on a half-line]\label{ex:1d}
We illustrate our result on the simplest example: 
\[\begin{cases}
u_t + H(u_x) -x^2 u_{xx}= 0, &  x >0, \\
-u_x =0, & x =0
\end{cases} \]
where $H$ is a quasi-convex function (i.e.
$\{p \in \R: H(p) \le \lambda \}$ convex for all $\lambda \in \R$) as
illustrated in Figure~\ref{fig:1}.  This example corresponds to the
case $N=1$ (number of branches) $d=0$ (dimension of the tangential
space) and $H_1 = H$ and
$L_{\mathrm{Neu}} (-u_t,\partial_1 u) = -\partial_1 u$.  In the three
pictures, the plain black curve represents the Hamiltonian $H=H_1$ and
the plain red curve represents the effective flux-limited function
$F_A (p_1) = \max (A,H^-(p_1))$ associated with the generalized flux
function $L_{\mathrm{Neu}}$ associated with the Neumann boundary
condition.  Depending on the position of the graph of $H$ and the
vertical line $\{p=0\}$, the effective flux limiter $A$ associated
with the Neumann condition varies.  On the left at the top, the line
$\{p=0\}$ intersects the graph of $H$ in its increasing part.  On the
right at the top, the vertical line $\{p=0\}$ intersects the graph of
$H$ in the non-decreasing part, but on a constant part. This second
case illustrates that we exhibit equivalent classes of boundary
conditions; indeed, different junction conditions can be equivalent to
the same effective one: other vertical lines (corresponding to
$u_x = \mathrm{const}$ at $x=0$) have the same effective boundary
condition (because they have the same effective flux limiter).  This
is also illustrated in the last case: the vertical line $\{p=0\}$ (and
others) intersects the graph of $H$ in its non-increasing part, which
implies that the flux limiter coincides with
$A_0=\min H = \max_i (\min H_i)$ for all the vertical lines appearing
in this third picture.
\end{ex}

\subsection{Comments on the main result}

Our main result, Theorem~\ref{thm:class}, extends the results
contained in \cite{im13,im14} in two directions: first, we can deal
with Kirchoff conditions (see Example~\ref{ex:kirchoff}), second we
can deal with second order terms (but degenerating along the
junction).

As in \cite{im13,im14}, the effective junction condition result is
quite a straightforward consequence of the following important fact
(Theorem~\ref{thm:reduced}): in order to check that a function is a
flux-limited sub- and super-solution, it is enough to use a reduced
set of test functions $\varphi$ whose normal derivatives
$\partial_i \varphi$ have specific values along $\Gamma$. For
instance, these normal derivatives are equal to $\pi_i^+(p',A(p'))$ if the
Hamiltonian has no constant parts and does not depend on $x'$. We
recall that, roughly speaking, $\pi_i^+$ is the inverse function of
the non-decreasing part of $H_i$, see \eqref{def:pi+} below.

The first version of this paper contained a comparison principle for
\eqref{eq:dp-gen}, under stronger assumptions on $F$. On the one
hand, the proof was quite difficult, relying on the vertex test
function introduced in \cite{im13,im14}, for which $C^2$ regularity
was to be proved in the multi-dimensional setting.  On the other hand,
new and simpler techniques now emerge to attack this problem, see for
instance \cite{bbcim,guerand,ls,monneau}. In particular, it is
explained in \cite{bbcim} that the equations considered in the present
work can be handled in the two-domain case.  For these two reasons, we
decided to restrict ourselves to the core of the work, that is to say
the study of effective junction conditions.

\subsection{Comments on assumptions}

Assumptions (F1), (L1), (L2) are natural (if not necessary) when
dealing with viscosity solutions of continuous Hamilton-Jacobi
equations. In particular, (L2) ensures that the junction condition is
compatible with the maximum principle.  We recall that our goal is to
exhibit effective junction conditions for degenerate parabolic
equations. In particular, we want to understand what are the
\emph{effective} junction conditions that are imposed at the
junction. From this point of view, it is necessary to consider
degenerate parabolic equations which actually degenerate along
$\Gamma$. This is exactly (F4). We also assume that the Hamiltonians
have convex sublevel sets, see (F4).  This condition can probably be
relaxed but until very recent contributions \cite{guerand,monneau,ls}
(none of these contributions were not available when the first version
of this work appeared), the non-convex case was out of reach. As far
as (F3) is concerned, it ensures that the Hamiltonians are coercive, a
property which is used repeatedly and is at the core of most
proofs. It is used together with (L4) for instance to derive the
``weak continuity'' of sub-solutions (see Lemma~\ref{lem:weak-cont}
below). Condition (F2) is used in an essential way when proving that
the set of test functions can be reduced (see the proof of
Lemma~\ref{second tec} about critical slopes below). Remark that this
condition is weaker than the one which is needed in order to prove
uniqueness, see \cite[Condition~(3.14)]{cil92}. To finish with, (L3)
and (L5) are used when proving the main result.

 \subsection{An application: the vanishing viscosity limit}

Because we are able to deal with Kirchoff conditions, we are in
position to adress an open problem about Hamilton-Jacobi equations
from ``regional control'' problem: the identification of the vanishing
viscosity limit.

 We study the limit as $\eps \to 0$ of the equation posed in
 $(0,+\infty) \times \R^{d+1}$
\begin{equation}\label{eq:vvl}
\begin{cases}
v^\eps_t + \tilde{H}_1(t,x,D v^\eps) = \eps \Delta v^\eps, & x_{d+1} <0, t>0 \\
v^\eps_t + \tilde{H}_2(t,x,D v^\eps) = \eps \Delta v^\eps, & x_{d+1} >0, t>0 \\
v^\eps(0,x) =v_0(x), & x \in \R^{d+1}
\end{cases}
\end{equation}
where $x = (x',x_{d+1}) \in \R^{d+1}$. 
In the previous equation, we do not need to impose any condition since
the Laplacian is strong enough to ensure the existence of solutions that
are continuously differentiable in the space variable $x \in \R^d$
despite the discontinuity of the first order term.  In particular, the
following condition holds at $x_{d+1}=0$,
\begin{equation}\label{e:k}
 \partial_{x_{d+1}} v^\eps (t,x',0+) = \partial_{x_{d+1}} v^\eps
(t,x',0-).
\end{equation}
In this specific singular perturbation problem, the limit is
identified by remarking that \eqref{e:k} is a Kirchoff condition and
that consequently we can pass to the limit using relaxed solutions;
more precisely, the limit of $v^\eps$ corresponds to a relaxed
solution associated with this specific generalized junction
condition. But the main theorem tells us that the limit thus
corresponds to a flux-limited solution associated with a flux limiter
$A$ that is explicitly given by a formula (see
Definition~\ref{defi:al}). Looking closely at this formula, we can prove
that it corresponds to the maximal Ishii solution of the limit
equation recently identified by Barles, Briani and Chasseigne
\cite{bbc1,bbc2}.
\begin{thm}[The vanishing viscosity limit selects the maximal Ishii
  solution]\label{thm:vvl-evol-tilde}
Assume
\[ \begin{cases}
  \tilde{H}_i \text{ continuous }\\
  \forall (t,x') \in (0,T) \times \R^d, \forall \lambda \in \R,
  \{ p=(p',p_i) \in \R^{d+1}: \tilde{H}_i (t,x',p) \le \lambda \} \text{ convex } \\
  \displaystyle \lim_{|p|\to +\infty} \inf_{(t,x') \in (0,T) \times
    \R^d}\tilde{H}_i (t,x',p) = +\infty
\end{cases}
\]
and $v_0$ is uniformly continuous in $\R^{d+1}$.  Let $v^\eps$ be
solution of \eqref{eq:vvl} such that there exists $C>0$ (independent
of $\eps$) such that $|v^\eps (t,x)-v_0(x)|\le Ct$ for all
$(t,x) \in (0,T) \times J$. Then $v^\eps$ converges towards the
maximal Ishii solution $v$ of
\begin{equation}\label{eq:hjtilde} 
\begin{cases} 
v_t + \tilde{H}_1 (t,x,Dv) = 0, &  x_{d+1} <0 , t>0\\ 
v_t + \tilde{H}_2 (t,x,Dv) = 0, &  x_{d+1} >0 , t>0
\end{cases} 
\end{equation}
subject to the initial condition
\[ v(0,x) = v_0 (x), \quad x \in \R^{d+1}. \]
\end{thm}
\begin{rem}
  The function $v$ is associated with the unique flux-limited solution
  $u$ of the previous Hamilton-Jacobi equation for some flux limiter
  $A_I^-(t,x',p')$ that was identified in a previous work (see
  \eqref{eq:ai-} in Proposition~\ref{prop:ishii} below,
  corresponding to \cite[Proposition~4.1]{im14}).  The functions $v$
  and $u$ satisfy the following equality:
  $v(t,x',x_{d+1}) = u (t,x',|x_{d+1}|)$, see
  Theorem~\ref{thm:vvl-evol} in Section~\ref{sec:vvl}.
\end{rem}

\subsection{Review of literature}

Semi-linear uniformly parabolic equations on compact networks were
studied in \cite{vb88,vbn96,kms07,pb04} where uniqueness, existence,
strong maximum principle among other results were proved to be true.

The first results for Hamilton-Jacobi equations on networks were
obtained in \cite{schieborn} for eikonal equations.  Some years later,
the results were extended in \cite{sc13,acct,imz13}. Many new results
were obtained since then, see for instance \cite{im13,im14} and references
therein. 

In \cite{bbc1,bbc2}, the authors study regional control, i.e. control
with dynamics and costs which are regular on either side of a
hyperplane but with no compatibility or continuity assumption along
the hyperplane. They identify the maximal and minimal Ishii solutions
as value functions of two different optimal control problems. They
also use the vanishing viscosity limit on a 1D example in order to
prove that the two Ishii solutions can be different. Moreover, the
authors ask about the vanishing viscosity limit in the general case.  

In \cite{cms13}, the authors study the vanishing viscosity limit
associated with Hamilton-Jacobi equations posed on a junction (the
simplest network, see above). The main difference with our results is
that the authors impose some compatibility conditions on Hamiltonians.
In particular, this allows them to construct viscosity solutions which
satisfy Kirchoff conditions in a strong sense. We proceed in a
different setting and in a different way: no compatibility conditions
on Hamiltonians are imposed, and Kirchoff conditions are understood in
a relaxed sense, which is stable under local uniform convergence (and
even relaxed semi-limits). We then use Theorem~\ref{thm:class} to
prove that imposing Kirchoff conditions reduce to the study of a
flux-limited problem (for which uniqueness holds true).

In his lectures at Coll\`ege de France \cite{lions-cdf}, Lions also
treats problems related to Hamilton-Jacobi equations with
discontinuities. After posting a first version of this paper, Lions
and Souganidis \cite{ls} wrote a note about a new approach for
Hamilton-Jacobi equations posed on junctions with coercive
Hamiltonians that are possibly not convex. 

We previously mentioned that, since the first version of this paper
were posted, Guerand and Monneau studied independently effective
non-linear boundary conditions in the non-convex case. On the one hand
Guerand \cite{guerand} studied the case $N=1$ in the 1D setting, which
amounts to studying first order non-convex Hamilton-Jacobi equations
with nonlinear boundary conditions of Neumann type.  On the other hand
Monneau \cite{monneau} mentioned to us that he studies effective
junction conditions for non-convex Hamilton-Jacobi equations posed on
multi-dimensional junctions.

As far as effective boundary conditions are concerned, we would like
to mention that there are some results for motion of interfaces by
Elliott, Giga and Goto \cite{egg} and for conservation laws by
Andreianov and Sbihi \cite{as1,as2,as3}.

To finish with, the link between the theory developed in
\cite{bbc1,bbc2} and flux-limited solutions from \cite{im13,im14} is
explored in \cite{bbcim}. In particular, \cite{bbcim} contains
alternative proofs in the two-domain case of the comparison principle
from \cite{im14} and of the vanishing viscosity limit obtained in the
present work.

\subsection{Organization of the paper}
In Section~\ref{sec:visc}, the notions of relaxed and flux-limited
solutions are presented and their properties studied. In
Section~\ref{sec:reduced}, it is proved that in order to check that a
function is a flux-limited solution, the set of test functions can be
reduced. In Section~\ref{sec:class}, we prove the main result of this
paper, Theorem~\ref{thm:class}. Section~\ref{sec:vvl} is devoted to
the study of the vanishing viscosity limit. The last section
(Section~\ref{sec:ld}) is devoted to the proof of a known result about
large deviations using the main result of this work.

\subsection{Notation}

A distance is naturally associated with the junction $J$: for
$x \in J_i$ and $y \in J_j$,
\[ d(x,y) = \begin{cases} |x'-y'| + |x_i - y_i| & \text{ if } i=j,\\ 
|x'-y'| + x_i + y_j & \text{ if } i \neq j.
 \end{cases} \] 
The open ball $B_r(t_0,x_0)$ centered
at $(t_0,x_0) \in \R \times J$ is defined as $(t_0-r,t_0+r) \times \{ y \in J: d(y,x_0) < r\}$.
\bigskip

The junction hyperplane $\Gamma$ is the common boundary of $J_i$: we
have $\Gamma = \partial J_i$. We identify $\Gamma$ with $\R^d$ and we
do not write the injection of $\R^d$ into $J_i$: $x' \mapsto (x',0)$.
For this reason, we write indisctinctively $x=(x',0) \in \Gamma$ and
$x' \in \Gamma$.  \bigskip

The Hamiltonian $H_i(t,x',p',p_i)$ is defined for $x' \in \Gamma$ and
$p \in \R^{d+1}$.  The minimal minimizer of
$p_i \mapsto H_i(t,x',p',p_i)$ is denoted by $\pi_i^0(t,x',p')$. The
functions $H_i^-$ and $H_i^+$ are defined as follows
\begin{eqnarray*}
 H_i^- (t,x',p',p_i) = \begin{cases} H_i (t,x',p',p_i) & \text{ if } p_i \le \pi_i^0(t,x',p') \\
H_i (t,x',p',\pi_i^0 (t,x',p')) & \text{ if } p_i \ge \pi_i^0(t,x',p')
\end{cases} \\
 H_i^+ (t,x',p',p_i) = \begin{cases} H_i (t,x',p',p_i) & \text{ if } p_i \ge \pi_i^0(t,x',p') \\
H_i (t,x',p',\pi_i^0 (t,x',p')) & \text{ if } p_i \le \pi_i^0(t,x',p').
\end{cases}
\end{eqnarray*}
\begin{figure}
\includegraphics[height=4cm]{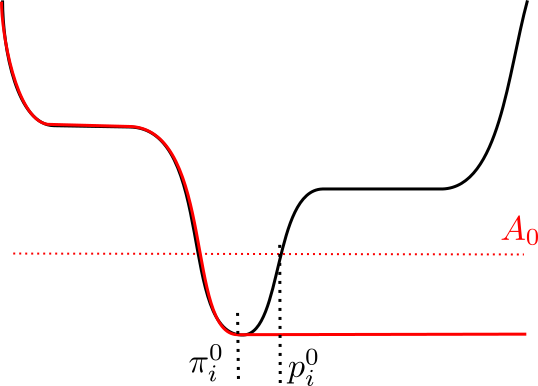} \hspace{4mm}
\caption{Non-increasing part $H_i^-$ of a Hamiltonian $H_i$. The
  Hamiltonian is in black, the monotone part in red. The tangent
  variables $(t,x',p')$ are not shown. In this example, the minimum of $H_i$ is lower than $A_0$. }
\label{fig:2} 
\end{figure}

\begin{figure}
\includegraphics[height=4cm]{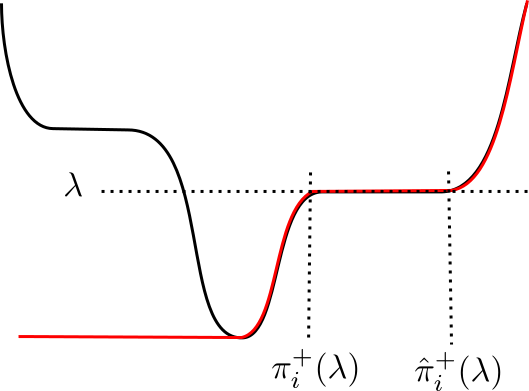}
\caption{Non-decreasing part $H_i^+$  of a Hamiltonian $H_i$. The
  Hamiltonian is in black, the monotone part in red. The tangent
  variables $(t,x',p')$ are not shown.}
\label{fig:3} 
\end{figure}

For $\lambda \ge \min_{p_i \in \R} H_i (t,x',p',p_i)$, the functions $\pi_i^+$ and $\hat{\pi}_i^+$ 
are defined by
\begin{eqnarray}
\label{def:pi+}
\pi_i^+ (t,x',p',\lambda) &= \inf \{  p_i : H_i (t,x',p',p_i) = H_i^+ (t,x',p',p_i) =
\lambda \}, \\
\label{def:pi+hat}
\hat{\pi}_i^+ (t,x',p',\lambda) & = \sup \{ p_i : H_i (t,x',p',p_i) = H_i^+
(t,x',p',p_i) = \lambda \}
\end{eqnarray}
The function $A_0$ is defined for $t,x',p' \in \R^d$ by \eqref{a0}. We recall that 
\[ A_0(t,x',p') = \max_{i = 1,\dots,N} \min_{p_i \in \R} H_i (t,x',p',p_i) . \]
The functions $p_i^0(t,x',p')$ are defined as 
\[ p_i^0 (t,x',p')= \pi_i^+ (t,x',p',A_0(t,x',p')).\]


\section{Relaxed and flux-limited solutions}
\label{sec:visc}

\subsection{Test functions}

In order to define relaxed and flux-limited solutions, the set of test
functions is to be made precise.
\begin{defn}[Test functions]\label{defi:test}
  A function $\phi : (0,T) \times J \to \R$ is a \emph{test function}
  for \eqref{eq:dp-gen} if it is continuous in $(0,T) \times J$,
  $\phi|_{(0,T) \times J_i}$ is $C^1_t \cap C^1_{x}$ and
  $\phi |_{(0,T) \times J_i^*}$ is $C^2_x$.
\end{defn}

We classically say that a function $\phi$ touches another function $u$
at a point $(t,x)$ from below (respectively from above) if $u \ge
\phi$ (respectively $u \le \phi$) in a neighbourhood of $(t,x)$ with
equality at $(t,x)$.

\subsection{Relaxed  solutions}

\begin{defn}[$L$-relaxed  solutions] \label{defi:relaxed} A function
  $u\colon (0,T) \times J \to \R$ is an  \emph{$L$-relaxed sub-solution}
  (resp. \emph{$L$-relaxed super-solution}) of \eqref{eq:dp-gen} if it is upper
  semi-continuous (resp. lower semi-continuous) and for all test functions 
  $\phi$ touching $u$ from above (resp. from below) at $(t,x) \in (0,T) \times J_i$, we have
\begin{eqnarray*}
& \phi_t + F_i(t,x,D\phi,D^2\phi) \le 0 \text{ at } (t,x) \\
(\text{resp. } & \phi_t + F_i(t,x,D\phi,D^2 \phi) \ge 0 \text{ at } (t,x) ) 
\end{eqnarray*}
if $x \notin \Gamma$, and
 \begin{eqnarray*}
 & \begin{cases}
 \text{either } \phi_t + H_i (t,x,D\phi) \le 0 \text{ at } (t,x) \text{ for some } i \in \{1,\dots,N\}, \\
 \text{or }  L (-\phi_t, \partial_1 \phi,\dots,\partial_N \phi,t,x,D'\phi) \le 0 \text{ at } (t,x)
 \end{cases} \\
 \Bigg(\text{resp. } 
 & \begin{cases}
 \text{either } \phi_t + H_i (t,x,D\phi) \ge 0 \text{ at } (t,x) \text{ for some } i \in \{1,\dots,N\}, \\
 \text{or } L (-\phi_t, \partial_1 \phi,\dots,\partial_N \phi,t,x,D'\phi) \ge 0 \text{ at } (t,x)
 \end{cases} \Bigg)
 \end{eqnarray*}
if $x \in \Gamma$. 
\end{defn}
The following observation is important for stability and the reduction
of the set of test functions. The proof contained in \cite{im13} can
be easily extended to generalized junction conditions. We give a
short proof for the reader's convenience.
\begin{lem}[``Weak continuity'' of relaxed sub-solutions]\label{lem:weak-cont}
  Assume \emph{(F)} and \emph{(L)}.  Let $u: (0,T) \times J \to \R$ be an $L$-relaxed
  sub-solution of \eqref{eq:dp-gen}. Then for all
  $i \in \{1,\dots,N\}$, and $x=(x',0) \in \Gamma$,
\[ u (t,x) = \limsup_{(s,y) \to (t,x), y \in J_i^*} u(s,y).\]
\end{lem}
\begin{proof}
Let $i \in \{1,\dots, N\}$. 
Since $u$ is upper semi-continuous, we have for all $(t,x) \in (0,T)\times \Gamma$,
\[ u (t,x) \ge \limsup_{(s,y) \to (t,x), y \in J_i^*} u(s,y) =:U_i(t,x).\]
Remark that the function $U_i:(0,T) \times \Gamma \to \R$ is upper semi-continuous. 
In order to prove that $u=U_i$ in $(0,T) \times \Gamma$, we assume that there exists
$(t_*,x_*) \in (0,T) \times \Gamma$ such that 
\begin{equation}\label{giga}
 u (t_*,x_*) \ge U_i(t_*,x_*) + \delta
\end{equation}
for some $\delta >0$. 

The density theorem \cite[Theorem~3.1]{clsw} can be applied to the
restriction of $-u$ to $(0,T) \times \Gamma$ around $(t_*,x_*)$.
Roughly speaking, this theorem claims that the proximal
subdifferential (which is a subset of the viscosity subdifferential)
is nonempty in a dense set. This result even ensures that there exists
a point $(t_0,x_0) \in (0,T) \times \Gamma$ such that
$(t_0,x_0,-u(t_0,x_0))$ is as close as desired to
$(t_*,x_*,-u (t_*,x_*))$ and there exists a viscosity subdifferential
of $-u$ at $(t_0,x_0)$.  More precisely, for all $\eps >0$, there
exists a $C^1$ function $\Psi : (0,T) \times \Gamma \to \R$ and
$(t_0,x_0) \in (0,T) \times \Gamma$ such that $\Psi$ strictly touches
$u$ from above at $(t_0,x_0) \in B_r (t_*,x_*) \cap (0,T) \times \Gamma$
for some $r>0$ and
\[ (t_0,x_0) \in B_\eps (t_*,x_*) \quad \text{ and } \quad 
-u (t_*,x_*) -\eps \le -u (t_0,x_0) \le -u (t_*,x_*).\]
In particular, $u (t_0,x_0) \ge u (t_*,x_*)$. 

Moreover, since $U_i$ is upper semi-continuous, we can choose $\eps$ small enough
in order to ensure that $U_i(t_0,x_0) \le U_i (t_*,x_*) + \delta/2$.  

We now get from \eqref{giga} that
\begin{equation}\label{star} 
  u (t_0,x_0) \ge \limsup_{(s,y) \to (t_0,x_0), y \in J_i^*} u(s,y) + \delta/2.
\end{equation}

Since the test function strictly touches $u$ at $(t_0,x_0)$, we have  $\Psi -u \ge \delta_1 >0$ in a
neighbourhood (with respect to $(0,T) \times J$) of
$\partial B_r(t_0,x_0) \subset (0,T) \times \Gamma$.  We now consider
the test function $\Phi (t,x) = \Psi (t,x') + p_j x_j$ for $x \in J_j$
with $p_j>0$ if $j \neq i$ and $p_i <0$. Thanks to \eqref{star},
$|p_i|$ can be chosen arbitrarily large.  We now use the coercivity of
the $F_j$ (see (F3)) to show that for $\min_j |p_j|$ large enough,
$\Phi$ touches $u$ from above at $(t_0,x_0)$. But this implies that
\[ L (-\partial_t \Phi (t_0,x_0),p_1,\dots,p_N,t_0,x_0',D' \Phi
(t_0,x_0)) \le 0\]
which contradicts (L4) since the $\min_{k =1,\dots, N} p_k = p_i \to -\infty$.  The
proof is now complete.
\end{proof}

\subsection{Stability and existence}

The following results related to stability of relaxed sub- and
super-solutions are expected; even more, relaxed solutions are defined 
in such a way that they satisfy such stability properties.  

In order to state the first stability result, we recall the definition
of upper semi-continuous envelope $u^*$ (resp. lower semi-continuous
envelope $u_*$) of a function $u: (0,T) \times J \to \R$,
\[ u^* (t,x) = \limsup_{(s,y) \to (t,x)} u(s,y), \quad 
 u_* (t,x) = \liminf_{(s,y) \to (t,x)} u(s,y).\]
\begin{prop}[Stability of relaxed solutions - I]\label{prop:stab1}
  Assume \emph{(F)} and \emph{(L)}. If $(u_\alpha)_\alpha$ is a family
  of relaxed sub-solutions (resp. relaxed super-solutions) of
  \eqref{eq:dp-gen} which is locally uniformly bounded from above
  (resp. from below), then the upper semi-continuous (resp. lower
  semi-continuous) envelope of $\sup_\alpha u_\alpha$
  (resp. $\inf_\alpha u_\alpha$) is a relaxed sub-solution
  (resp. relaxed super-solution) of \eqref{eq:dp-gen}.
\end{prop}
\begin{proof}
We only treat the sub-solution case since the super-solution one is similar. 
  Let $u$ denote the upper semi-continuous envelope of
  $\sup_\alpha u_\alpha$.  Consider a test function $\phi$ strictly
  touching $u$ from above at $(t,x)$. There then exist a sequence
  $(t_n,x_n) \to (t,x)$ and $\alpha_n$ such that $\phi$ touches
  $u_{\alpha_n}$ from above at $(t_n,x_n)$. Writing the viscosity
  inequalities and passing to the limit yields the desired result.
\end{proof}
In order to state the second stability result, we recall the
definition of upper semi-limit $\bar{u}$ (resp. lower semi-limit
$\underline{u}$) of a family of functions
$u^\eps: (0,T) \times J \to \R$, $\eps>0$,
\[ \bar{u} (t,x) = \limsup_{(s,y) \to (t,x),\eps \to 0} u^\eps(s,y), \quad 
 \underline{u} (t,x) = \liminf_{(s,y) \to (t,x), \eps \to 0} u^\eps(s,y).\]
\begin{prop}[Stability of relaxed solutions - II]\label{prop:stab2}
  Assume \emph{(F)} and \emph{(L)}. If $\{ u^\eps\}_{\eps>0}$ is a
  family of relaxed sub-solutions (resp. relaxed super-solutions) of
  \eqref{eq:dp-gen} which is locally uniformly bounded from above
  (resp. from below), then the relaxed upper limit (resp. relaxed
  lower limit) of $\{u^\eps \}_{\eps >0}$ is a relaxed sub-solution (resp. relaxed
  super-solution) of \eqref{eq:dp-gen}.
\end{prop}
\begin{proof}
  We only treat the sub-solution case since the super-solution one is
  similar.  Consider a test function $\phi$ strictly touching $\bar u$
  from above at $(t,x)$. We can assume that the contact is
  strict. There then exist a sequence $(t_k,x_k) \to (t,x)$ and
  $\eps_k\to 0$ such that $\phi$ touches $u_{n_k}$ from above at
  $(t_k,x_k) \to (t,x)$ as $k \to + \infty$. Either there is a
  subsequence $k_p$ along which $x_{k_p} \in J_i^*$ for some
  $i \in \{1,\dots, N\}$ or $x_k \in \Gamma$ for large $k$'s. Writing
  the viscosity inequalities in both cases and passing to the limit
  yields the desired result.
\end{proof}
The stability properties satisfied by relaxed solutions ensure the
existence of discontinuous relaxed solutions. 
\begin{thm}[Existence of discontinuous relaxed solutions]\label{thm:existence}
  Assume \emph{(F)} and \emph{(L)} and consider $u_0$ uniformly
  continuous. Assume also that for all $R>0$, 
\[ C_R := \sup\{  |F_i(t,x,p,A)| : {i \in \{1,\dots,N\}, t \in (0,T), x \in J, |p|\le R, |A|\le R}\}<+\infty.\]
There exists $u$ such that its upper semi-continuous
  (resp. lower semi-continuous) envelope is a relaxed
  sub-solution (resp. relaxed super-solution) of
  \eqref{eq:dp-gen} such that 
\[ u (0,x) = u_0 (x) \text{ for } x \in J .\]
\end{thm}
\begin{rem}
This theorem states the existence of discontinuous solutions in the
sense of Ishii \cite{ishii87}. 
\end{rem}
\begin{proof}
  In view of the stability results, it is enough to construct a
  solution for some initial datum $u_0$ such that $u^i_0=u_0|_{J_i}$
  are in $C^{1,1}$. For such $u_0$'s, we can construct barriers in the
  classical way: $u^\pm (t,x) = u_0 (x) \pm Ct$.  For $C \ge C_{R_0}$
  with $R_0 \ge \|Du_0^i\|_\infty + \|D^2u_0^i\|_\infty$ for all
  $i =1,\dots,N$, the function $u^+$ is a relaxed super-solution while
  $u^-$ is a relaxed sub-solution. Indeed, as far as the equations in
  $J_i$ are concerned, it is classical; as far as the junction
  condition is concerned, the equation is satisfied up to $\Gamma$ and
  thus $u^\pm$ are relaxed semi-solutions on $\Gamma$.  We then
  consider $W$ the set of all functions lying below $u^+$ whose upper
  semi-continuous envelope is a relaxed sub-solution. Then the
  supremum of $w \in W$ is in $W$ and it is maximal. Let $w$ denote
  this maximal element. If the lower semi-continuous envelope is not a
  relaxed super-solution, there exists a test function $\phi$ and a
  point $(t,x)$ such that $\phi$ touches $w_*$ from below at $(t,x)$
  without satisfying the corresponding viscosity inequality. This
  implies $\phi < (u_+)_*$ in a neighbourdhood of $(t,x)$ and we can
  prove that $\phi$ is a relaxed sub-solution in the same
  neighbourhood. Then we can construct a relaxed sub-solution
  $w_\delta$ which is not below $w$, contradicting its maximality.
\end{proof}

\subsection{Flux-limited solutions}

It is proved in \cite{im13} that, in the special case where $L=L_A$
defined in \eqref{eq:defi-la} and for first order Hamilton-Jacobi
equations, relaxed solutions satisfy the junction condition in a
strong sense, which is made precise in the following definition.
\begin{defn}[Flux-limited solutions]\label{defi:FL}
  Given a function $A:(0,T) \times \Gamma \times \R^d \to \R$ such
  that $A \ge A_0$, a function $u\colon (0,T) \times J \to \R$ is a
  \emph{$A$-flux-limited sub-solution} (resp. \emph{$A$-flux-limited
    super-solution}) of \eqref{eq:dp-gen} if it is upper
  semi-continuous (resp. lower semi-continuous) and for any test
  function $\phi$ in the sense of Definition~\ref{defi:test} touching
  $u$ from above (resp. from below) at $(t,x) \in (0,T) \times J_i$,
  we have
\begin{eqnarray*}
& \phi_t + F_i(t,x,D\phi,D^2\phi) \le 0 \text{ at } (t,x) \\
\bigg(\text{resp. } & \phi_t + F_i(t,x,D\phi,D^2 \phi) \ge 0 \text{ at } (t,x) \bigg) 
\end{eqnarray*}
if $x \notin \Gamma$, and
 \begin{eqnarray*}
 & L_A (-\phi_t,\partial_1 \phi,\dots,\partial_N \phi,t,x', D' \phi) \le 0 \text{ at } (t,x) \\
 \bigg(\text{resp. } 
& L_A (-\phi_t,\partial_1 \phi,\dots,\partial_N \phi,t,x', D' \phi)\ge 0 \text{ at } (t,x) \bigg)
 \end{eqnarray*}
if $x \in\Gamma$. 
\end{defn}
\begin{rem} \label{rem:one-point} When proving that a function is a
  sub-solution or a super-solution of \eqref{eq:dp-gen} at one given
  point of $(0,T) \times \Gamma$, it is enough to consider a reduced
  set of test functions associated with this specific point. It is
  thus interesting to consider sub- or super-solution of
  \eqref{eq:dp-gen} \emph{at only one point of $(0,T) \times \Gamma$}
  -- see Theorem~\ref{thm:reduced} about the reduction of the set of
  test functions.
\end{rem}
The following proposition asserts that $L_A$-relaxed solutions
coincide with $A$-flux-limited solutions. It was proved in
\cite{im13,im14} in the case of first order equations. We point out
that the multidimensional proof of \cite{im14} applies \emph{without
  any change} to degenerate parabolic equations satisfying (F).
\begin{prop}[$L_A$-relaxed solutions are $A$-flux-limited solutions --
  \cite{im14}]\label{prop:rel-fl}
  Assume \emph{(F)} and \emph{(L)}. Then any $L_A$-relaxed sub-solution
  (resp. super-solution) of \eqref{eq:dp-gen} is an $A$-flux-limited
  sub-solution (resp. super-solution) of \eqref{eq:dp-gen}.
\end{prop}

\section{Reduced set of test functions for flux-limited solutions}
\label{sec:reduced}

In this section, we explain why it is sufficient to consider a reduced
set of test functions in order to check that a function is a
flux-limited (sub/super)solutions of \eqref{eq:dp-gen}. Such a result
is used in an essential way when proving Theorem~\ref{thm:class}.

\begin{defn}[Reduced test functions]\label{defi:test-reduced}
  Consider a flux limiter $A \ge A_0$  and a
  point $(t_0,x_0') \in (0,T) \times \Gamma$.  A function
  $\varphi : (0,T) \times J \to \R$ is a \emph{reduced test function}
  for \eqref{eq:dp-gen} at $(t_0,x_0')$ if there exists a function
  $\phi \in C^1((0,T) \times \R^d)$ and $N$ functions
  $\phi_i \in C^1 ([0,+\infty))$, $i=1,\dots,N$, such that
  \[\forall t \in (0,T), \forall  (x',x_i) \in J_i, \qquad \varphi (t,(x',x_i)) = \phi (t,x') + \phi_i (x_i) \]
  and, for all $i=1,\dots,N$, $\phi_i(0)=0$ and the slope $p_i= \phi_i'(0)$ and
  the tangential gradient $p'=D'\phi (t_0,x'_0)$ satisfy
\begin{equation}
\label{eq:reduced-slope}
 H_i (t_0,x'_0,p',p_i) = H_i^+ (t_0,x'_0,p',p_i) = A (t_0,x'_0,p')
\end{equation}
that is to say $p_i \in [\pi_i^+(t_0,x'_0,p'),\hat{\pi}_i^+(t_0,x'_0,p')]$. 
\end{defn}
Theorem~\ref{thm:reduced} below generalizes the one contained in
\cite{im13}. In order to state it, we need to consider the equation on
each (open) branch $i$, i.e. away from the junction hyperplane $\Gamma$:
\begin{equation} \label{eq:dp-far}
u_t +  F_i(t,x,Du,D^2u) =0, \quad (t,x)\in (0,T) \times J_i^*.
\end{equation}
We can now state and prove the following theorem. 
\begin{thm}[Reducing the set of test functions]\label{thm:reduced}
  Assume \emph{(F)} and consider a function
  $A:(0,T) \times \Gamma \times \R^d \to \R$ such that $A \ge A_0$.  Given a
  function $u:(0,T) \times J \to \R$, the following properties hold
  true.
\begin{enumerate}[i)]
\item If, for all $i\in \{1,\dots,N\}$,
  $u$ is a sub-solution of \eqref{eq:dp-far} and for 
  $(t,x) \in (0,T) \times \Gamma$,
\begin{equation}\label{eq:cont-weak}
 u(t,x) = \limsup_{s \to t, y \to x, y \in J_i^*} u(s,y),
\end{equation} 
then $u$ is an $A_0$-flux limited
sub-solution of \eqref{eq:dp-gen} at $(t,x)$.
\item If, for all $i\in \{1,\dots,N\}$, $u$ is a sub-solution of \eqref{eq:dp-far} satisfying
  \eqref{eq:cont-weak} and if for any reduced test function $\varphi$
  in the sense of Definition~\ref{defi:test-reduced} touching $u$ from
  above at $(t, x) \in (0,T) \times \Gamma$, we have
\begin{eqnarray*}
\varphi_t (t,x)+ A (x',D'\varphi (t,x)) \le 0,
\end{eqnarray*}
then $u$ is an $A$-flux-limited sub-solution  of \eqref{eq:dp-gen}  at $(t, x)$.
\item If, for all $i\in \{1,\dots,N\}$,  $u$ is a super-solution of
  \eqref{eq:dp-far} and if for any reduced test function $\varphi$ in
  the sense of Definition~\ref{defi:test-reduced} touching $u$ from
  below at $(t, x) \in (0,T)\times \Gamma$ we have
  \[ \varphi_t (t,x)+ A (x,D'\varphi (t,x)) \ge 0, \]
then $u$ is an $A$-flux-limited super-solution of \eqref{eq:dp-gen} at $(t, x)$.
\end{enumerate}
\end{thm}
\begin{rem}
In the previous statement, functions are flux-limited solution of \eqref{eq:dp-gen} at
only one point of $(0,T) \times \Gamma$ -- see
Remark~\ref{rem:one-point} above.
\end{rem}
\begin{proof}
  The proof of \cite[Theorem~2.7]{im13} applies here \emph{without any
    change} after proving the two lemmas~\ref{first tec} and
  \ref{second tec} about \emph{critical normal slopes}.  Indeed, with
  such technical results in hands, the proof focuses on what happens
  on $\Gamma$ and second derivatives do not appear any more.
\end{proof}
\begin{lem}[Super-solution property for the critical normal slope on each branch]\label{first tec} 
  Let $i \in \{1,\dots,N\}$ be fixed.  Let
  $u : (0, T)\times J_i \to \R$ be a lower semi-continous
  super-solution of \eqref{eq:dp-far}. Let $\phi$ be a test function
  touching $u$ from below at some point
  $(t_0, x_0) \in (0, T) \times \Gamma$. We consider
\[
\overline{p_i} = \sup\{\overline{p} \in \R : 
    \exists r > 0, \phi(t, x)+\overline{p}x_i \le u(t, x) 
 \text{ for }    (t, x) \in  B_r (t_0,x_0) \cap (0,T) \times J_i\}.
\]
If $\overline{p_i} < +\infty$, then we have 
\[
\phi_t+ H_i(t,x,D'\phi, \partial_i \phi+\overline{p_i} ) \quad \ge 0 \quad \text{ at } (t_0,x_0)
\quad \text{ with } \overline{p_i} \ge 0.
\]
\end{lem}
\begin{lem}[Sub-solution property for the critical normal slope on each branch]\label{second tec}  
  Let $i \in \{1,\dots ,N\}$ be fixed.  Let
  $u : (0, T)\times J_i \to \R$ be a sub-solution of
  \eqref{eq:dp-far}. Let $\phi$ be a test function touching $u$ from
  above at some point $(t_0, x_0) \in (0, T) \times \Gamma$. We
  consider
\[
\underline{p_i} = \inf\{\overline{p} \in \R : 
    \exists r > 0, \phi(t, x)+\overline{p}x_i \ge u(t, x) 
 \text{ for }  (t, x) \in   B_r (t_0,x_0) \cap (0,T) \times J_i\}.
\]
If $u$ satisfies
\begin{equation}\label{eq:cont-weak-i}
 u(t_0,x_0) = \limsup_{s \to t_0, y \to x_0, y \in J_i^*} u(s,y),
\end{equation}
then $\underline{p_i} > - \infty$; moreover, 
 we have in this case
\[
\phi_t+  H_i(t,x,D'\phi,\partial_i \phi+\underline{p_i} )\le 0 
\quad \text{ at } (t_0,x_0) \quad  \text{ with } \underline{p_i}  \le 0.
\]
\end{lem}
We first prove Lemma~\ref{first tec}.
\begin{proof}[Proof of Lemma~\ref{first tec}]
The proof follows the same lines of \cite[Lemma 2.8]{im13}. 

From the definition of $\overline{p_i}$, for all $\e>0$ small enough,
there exists $\delta=\delta(\e) \in (0,\e)$ such that
\begin{equation}\label{eq:dessous}
 u(s, y) \ge  \phi(s, y)+(\overline{p_i}-\e) y_i \quad  
\text{ for all } (s,y) \in B_\delta(t_0,x_0) \cap (0,T) \times J_i
\end{equation}
and there exists $(t_\e,x_\e) \in B_{\delta/2}(t_0,x_0)$ such that
\begin{eqnarray*}
 u(t_\e,x_\e)  <  \phi(t_\e,x_\e) +(\overline{p_i}+\e)x_\e^i.
\end{eqnarray*}
We choose a smooth function $\Psi:\R^{d+2} \to [-1,0]$ such that
\[
\Psi= \begin{cases}
0 & \text{ in } B_{1/2}(t_0,x_0) \\
-1 & \text{ outside } B_{1}(t_0,x_0).
\end{cases}
\]
We define for $(s,y) \in (0,T) \times J_i$, 
\[
\Phi(s,y)=\phi(s,y)+2 \e \Psi_\delta(s,y)+(\overline{p_i}+\e)y_i 
\]
with 
\(
\Psi_\delta(Y)=\delta \Psi\left(\frac{Y}{\delta}\right).
\)
Remark that for $(s,y) \in \partial (B_\delta(t_0,x_0)\cap (0,T) \times J_i)$, 
we have $y_i \le \delta$. In particular,
$-2 \e \delta+(\overline{p_i}+\e)y_i \le (\overline{p_i}-\e)y_i$ for such $(s,y)$. 
Hence  \eqref{eq:dessous} implies  
\[
\begin{cases}
\Phi(s,y)=\phi(s,y)-2  \e \delta+(\overline{p_i}+\e)y_i \le u(s,y) & 
\text{ for } (s,y) \in \partial (B_\delta(t_0,x_0) \cap (0,T) \times J_i), \medskip \\
\Phi(s,x) \le \phi(s,x) \le u(s,x) & \text{ for } (s,x) \in (t_0-\delta,t_0+ \delta) \times \Gamma
\end{cases}
\]
 and
\[
\Phi(t_\e,x_\e)=\phi(t_\e,x_\e)+(\overline{p_i}+\e)x_\e^i > u(t_\e,x_\e).
\]
We conclude that there exists a point
$(\overline{t_\e},\overline{x_\e}) \in B_\delta(t_0,x_0) \cap ((0,T)
\times J_i^*)$
such that $u-\Phi$ reaches a minimum in
$\overline{B_\delta(t_0,x_0)} \cap ([0,T] \times J_{i})$.
We thus can write the viscosity inequality
\begin{eqnarray*}
\Phi_t+  F_i(t,x,D\Phi,D^2\Phi) \ge 0 \quad \text{ at } (\overline{t_\e},\overline{x_\e})
\end{eqnarray*}
which reads
\begin{multline}
\label{eq:eps}
  \phi_t(\overline{t_\e},\overline{x_\e}) +2 \e (\Psi_\delta)_t(\overline{t_\e},\overline{x_\e}) \\
  + F_i(\overline{t_\e},\overline{x_\e}, (D'  \phi+2\eps D' \Psi_\delta) (\overline{t_\e},\overline{x_\e}), \partial_i \phi
  (\overline{t_\e},\overline{x_\e}) +2\e \partial_i 
\Psi_\delta(\overline{t_\e},\overline{x_\e})+\overline{p_i}+\e,
  D^2 \phi+2\e  D^2 \Psi_\delta(\overline{t_\e}, \overline{x_\e})) \ge 0.
\end{multline}
We now send $\e \to 0$ in the above inequality; recall that
$\delta \in (0,\eps)$ and $\Psi_\delta = \delta \Psi (\cdot /\delta)$;
in particular,
\begin{equation}\label{eq:eps2}
 \e (\Psi_\delta)_t(\overline{t_\e},\overline{x_\e}), \eps D'
\Psi_\delta(\overline{t_\e},\overline{x_\e}), \e \partial_i
\Psi_\delta(\overline{t_\e},\overline{x_\e}) \to 0 \quad \text{ as }
\eps \to 0.
\end{equation}
As far as second derivatives are concerned, we have 
\[ | \eps D^2 \Psi_\delta | \le \|D^2 \Psi\|_\infty.\]
In particular, 
\begin{equation}\label{eq:eps3}
 \e  D^2 \Psi_\delta(\overline{t_\e}, \overline{x_\e})) \to B \in \Sd
\end{equation}
along a subsequence. Since $(\overline{t_\e},\overline{x_\e}) \to (t_0,x_0)$, we finally
get from \eqref{eq:eps}, \eqref{eq:eps2} and \eqref{eq:eps3} that 
\[ \phi_t(t_0,x_0) + F_i (t_0,x_0,D'\phi (t_0,x_0), \partial_i \phi
(t_0,x_0) + \overline{p}_i, D^2 \phi (t_0,x_0)+B) \ge 0 \]
which is the desired inequality since $x_0 \in \Gamma$ and $F_i$ satisfies (F4). The proof is now complete.
\end{proof}
We now turn to the proof of Lemma~\ref{second tec}
\begin{proof}[Proof of Lemma~\ref{second tec}]
The main difference with the previous lemma is the claim that the critical normal slope is finite. 
This is the reason why we only explain this point. Here again, we follow closely \cite{im13}. 

Let $p \in (-\infty,0]$ be such that there exists $r >0$
such that $\phi + p  x_i \ge u$ in
$B = B_{r} (t_0,x_0) \cap (0,T) \times J_i$.  Remark first that,
replacing $\phi$ with $\phi+ (t-t_0)^2+|x-x_0|^2$ if necessary, we can
assume that
\begin{equation}\label{eq:strict} 
u(t,x) < \phi(t,x) + p  x_i \text{ if } (t,x) \neq (t_0,x_0).
\end{equation}
 In particular, there exists $\delta >0$ such
that $\phi +p  x_i \ge u + \delta$ on 
$\partial B \setminus \Gamma$.

Since $u$ satisfies \eqref{eq:cont-weak-i}, there exists $(t_\eps,x_\eps) \to (t_0,x_0)$ such 
that $x_\eps \in J_i^*$ and $u(t_0,x_0) = \lim_{\eps \to 0} u (t_\eps,x_\eps)$. 

We now introduce the following perturbed test function 
\[ \Psi (t,x) = \phi (t,x) + p  x_i + \frac{\eta}{x_i} \]
where $\eta=\eta (\eps)$ is a small parameter to be chosen later. 
Let $(s_\eps,y_\eps)$ realize the infimum of $\Psi-u$ in $\overline{B}$. In particular,
\begin{equation}\label{estim:pen}
(\phi +p  x_i -u)(s_\eps,y_\eps) \le  \Psi (s_\eps,y_\eps) - u(s_\eps,y_\eps) 
\le  \Psi (t_\eps,x_\eps) - u(t_\eps,x_\eps) \to 0 \quad \text{ as } \quad \eps \to 0
\end{equation}
as soon as $\eta(\eps) = o (x_\eps^i)$ with
$x_\eps =(x'_\eps,x^i_\eps)$. In particular, in view of
\eqref{eq:strict}, this implies that $(s_\eps,y_\eps) \to (t_0,x_0)$ as
$\eps \to 0$. Since $u$ is a sub-solution of \eqref{eq:dp-far}, we know
that
\[ \phi_t (s_\eps,y_\eps) + F_i (s_\eps,y_\eps,D'\phi (s_\eps,y_\eps),\partial_i \phi (s_\eps,y_\eps) + p  -
\frac{\eta}{(y_\eps^i)^2} , D^2\phi(s_\eps,y_\eps) +
\frac{2\eta}{(y_\eps^i)^3} e_{d+1} \otimes e_{d+1}) \le 0
\]
(where $(e_1,\dots,e_{d+1})$ is an orthonormal basis of $\R^{d+1}$ and $e_{d+1}$ is orthogonal to $\Gamma$). 
Use now (F2) in order to get
\[ \phi_t (s_\eps,y_\eps) + F_i (s_\eps,y_\eps,D'\phi (s_\eps,y_\eps),\partial_i \phi (s_\eps,y_\eps) + p  -
\frac{\eta}{(y_\eps^i)^2} , D^2\phi(s_\eps,y_\eps) ) \le 2C_i \frac{\eta}{y_\eps^i}.
\]
Remark now that \eqref{estim:pen} implies
\[ \frac{\eta}{y_\eps^i} - \frac{\eta}{x_\eps^i} \le \left(p  (x_\eps^i-y_\eps^i) + (u-\phi)(s_\eps,y_\eps) - (u-\phi) (t_\eps,x_\eps)\right) \to 0 \quad \text{ as } \eps \to 0. \]
Recalling that $\eta$ is chosen so that $\eta / x_\eps^i \to 0$ as $\eps \to 0$, we thus get
\[ \frac{\eta}{y_\eps^i} \to 0 \quad \text{ as } \eps \to 0.\]
In particular, the coercivity of $F_i$ (see (F3)) implies that
$p  - \frac{\eta}{(y_\eps^i)^2}$ is bounded as $\eps \to 0$. Hence we
can pass to the limit as $\eps \to 0$ in the viscosity inequality and
get
\[ 
\phi_t (t_0,x_0) + H_i (t_0,x_0,D'\phi (t_0,x_0), \partial_i \phi (t_0,x_0) + p ^0) \le 0
\]
where $p ^0 \in (-\infty,0]$ is any accumulation point of
$p  - \frac{\eta}{(y_\eps^i)^2}$ as $\eps \to 0$. The previous
inequality and (F3) implies in particular that $p ^0$ is bounded from below by
a constant $C$ which only depends on $H_i, \phi_t,D\phi$ at
$(t_0,x_0)$. Indeed, (F3) implies in particular that 
\[ \lim_{|p|\to +\infty} \inf_{(t,x') \in (0,T) \times \Gamma}
H_i(t,x',p) = +\infty. \]
But this also implies that $p \ge C$ and, in turn, $\underline{p_i} \ge C$. The proof is now complete.
\end{proof}

\section{Proof of the main theorem}
\label{sec:class}

This section is devoted to the proof of the first main result,
Theorem~\ref{thm:class}. Throughout this section, we do not write the
$(t,x',p')$ dependence of $A_L$, $\pi^+$, $\hat{\pi}^+$ etc. (see
\eqref{def:pi+} and \eqref{def:pi+hat} for a definition) in order to
clarify the presentation and proofs. 

The proof of Theorem~\ref{thm:class} relies on properties and other
representations of the effective flux limiter $A_L$; we gather them
in the following preparatory proposition.
\begin{prop}[Representations of $A_L$]\label{prop:al}
Let $A_L$ be the effective flux limiter given by  Definition~\ref{defi:al}. 
\begin{enumerate}[i)]
\item If $L(A_0, \pi_1^+(A_0), \dots, \pi_N^+ (A_0)) \le 0$ then $A_L = A_0$. 
\item If $L(A_0, \pi_1^+(A_0), \dots, \pi_N^+ (A_0)) >0$ then $A_L$ is
  well defined: there exists a unique $\lambda^* \in \R$ and
  there exist
  $p_i^* \in [\pi_i^+(\lambda^*),\hat{\pi}_i^+ (\lambda^*)]$ (not
  necessarily unique) such that
  \( L(\lambda^*,p_1^*,\dots,p_N^*) = 0.\)
\item If $L(A_0, \pi_1^+(A_0), \dots, \pi_N^+ (A_0)) >0$, then
\begin{eqnarray}
\label{eq:al}
 A_L =  \sup \{ \lambda \ge A_0 : L(\lambda,\pi_1^+(\lambda),\dots,\pi_N^+ (\lambda)) > 0 \}   \\
\label{eq:al+}
A_L = \inf \{ \lambda \ge A_0 : L (\lambda,\hat{\pi}_1^+ (\lambda), \dots,
\hat{\pi}_N^+ (\lambda)) < 0 \}.
\end{eqnarray}
\item Moreover, if $L(A_0, \pi_1^+(A_0), \dots, \pi_N^+ (A_0)) >0$, we also have
\begin{eqnarray}\label{view:def}
 L (A_L,\pi_1^+(A_L),\dots,\pi_N^+(A_L)) \ge 0 \\
\label{view:def+}
 L (A_L, \hat{\pi}_1^+ (A_L), \dots,\hat{\pi}_N^+ (A_L)) \le 0.
\end{eqnarray}
\end{enumerate}
\end{prop}
\begin{rem}
  We point out that
  $p_i^* \in [\pi_i^+(\lambda),\hat{\pi}_i^+ (\lambda^*)$ is equivalent
  to $p_i^* \ge p_i^0$ and $L(\lambda^*,p_1^*,\dots,p_N^*)=0$.
\end{rem}
\begin{proof}
  Remark that $p_i^0$ in Definition~\ref{defi:al} coincides with
  $\pi_i^+(A_0)$.  In particular, if
  \[L(A_0, \pi_1^+(A_0), \dots, \pi_N^+ (A_0)) \le 0\] then
  Definition~\ref{defi:al} says that $A_L = A_0$. This proves i). 

We now assume that $L(A_0, \pi_1^+(A_0), \dots, \pi_N^+ (A_0)) >0$.
Assumption~(L5) implies that there exists $\tilde \lambda > A_0$ such that $L(\tilde \lambda, 
\hat{\pi}_1^+(\tilde \lambda), \dots, \hat{\pi}_N^+ (\tilde \lambda)) < 0$. 
In particular, the two following quantities are finite, 
\begin{eqnarray*}
S:= \sup \{ \lambda \ge A_0 : L(\lambda,\pi_1^+(\lambda),\dots,\pi_N^+ (\lambda)) > 0 \} 
\\
I:=\inf \{ \lambda \ge A_0 : L (\lambda,\hat{\pi}_1^+ (\lambda), \dots,
\hat{\pi}_N^+ (\lambda)) < 0 \}.
\end{eqnarray*}
Using that $\pi_i^+$ is left continuous and $\hat{\pi}_i^+$ is right continuous,
we have 
\begin{eqnarray}
\label{eq:s}
 L (S, \pi_1^+(S), \dots, \pi_N^+ (S)) \ge 0 \\
\label{eq:i}
 L(I, \hat{\pi}_i^+ (I), \dots, \hat{\pi}_N^+ (I)) \le 0.
\end{eqnarray}
Proving ii), iii) and iv) (apart from uniqueness in ii)) reduces to
proving that $S=I$. Indeed, if $S=I$ then: iii) is proved with
$A_L=I=S$; \eqref{view:def} and \eqref{view:def+} are satisfied with
$A_L=I=S$; the continuity of $L$ (see (L1)) and the two previous
inequalities imply the existence of
$p_i^* \in [\pi_i^+(A_L),\hat{\pi}_i^+ (A_L)]$ such that
$L (A_L,p_1^*,\dots,p_N^*) = 0$.

If $I < S$, then $\hat{\pi}_i^+ (I) < \pi_i^+ (S)$ for all
$i \in \{1,\dots, N\}$; but (L3) and \eqref{eq:s} then imply that
\[ L(I, \hat{\pi}_1^+ (I), \dots, \hat{\pi}_N^+ (I)) >  L (S, \pi_1^+(S), \dots, \pi_N^+ (S)) \ge 0 \]
which contradicts \eqref{eq:i}. Then $S \le I$. 

If $S < I$ then the definitions of $S$ and $I$ imply that for all $\lambda^* \in ]S,I[$, 
\begin{eqnarray*}
 L (\lambda^*, \pi_1^+(\lambda^*), \dots, \pi_N^+ (\lambda^*)) \le 0 \\
 L(\lambda^*, \hat{\pi}_1^+ (\lambda^*), \dots, \hat{\pi}_N^+ (\lambda^*)) \ge 0.
\end{eqnarray*}
But using the continuity of $L$ (see (L1)), this implies that for all $\lambda^* \in ]S,I[$, 
there exist $p_i^* \in [\pi_i^+(\lambda^*),\hat{\pi}_i^+(\lambda^*)]$, $i=1,\dots, N$, such that 
\[ L(\lambda^*,p_1^*, \dots, p_N^*) =0.\]
But this cannot be true for two different $\lambda^*$'s because of
(L3). Hence $S=I$. Notice that we can prove in the same way uniqueness
in ii). The proof is now complete.
\end{proof}

We now prove the main theorem.
\begin{proof}[Proof of Theorem~\ref{thm:class}] 
  Let $A_L$ be the effective flux limiter in the sense of
  Definition~\ref{defi:al}. It is well defined thanks to
  Proposition~\ref{prop:al}. Since $A_L \ge A_0$, the coercivity is
  clear:
  $\lim_{|p|\to +\infty} \inf_{x' \in \Gamma} A(t,x',p') =+\infty$.  The
  proof of the continuity of $A_L$ and the convexity of sublevel sets
  is the same as in \cite[Proof of Theorem~2.13]{im14}.

  We only deal with the sub-solution case since the super-solution
  case is very similar. If $A_L(t,x',p')=A_0(t,x',p')$, then
  Lemma~\ref{lem:weak-cont} and Theorem~\ref{thm:reduced} imply that
  any $L$-relaxed sub-solution of \eqref{eq:dp-gen} is an $A_0$-flux
  limited sub-solution of \eqref{eq:dp-gen}.

  We now consider the case where there exists $(t,x',p')$ such that
  $A_L(t,x',p') >A_0(t,x',p')$.  Let $u$ be an $L$-relaxed
  sub-solution of \eqref{eq:dp-gen} and let us prove that it is an
  $(A_L-\eps)$-flux-limited sub-solution of \eqref{eq:dp-gen} at
  $(t,x') \in (0,T)\times \Gamma$ for all $\eps >0$ such that $A_L-\eps >A_0$ (at
  $(t,x',p')$). We use here the fact that Theorem~\ref{thm:reduced} is
  local in the sense that it asserts that a function is a flux-limited
  solution at one given point $(t,x')\in (0,T) \times \Gamma$. In view of
  Lemma~\ref{lem:weak-cont} and Theorem~\ref{thm:reduced}, we only
  have to consider a reduced test function $\varphi$ touching $u$ from
  above at $(t,x') \in (0,T) \times \Gamma$. We recall that
\[ \varphi (s,y) = \phi (s,y') + \phi_i (y_i)\]
with $\phi_i (0)=0$ and $\phi_i' (0)  \in [\pi_i^+ (A_L-\eps),\hat{\pi}_i^+ (A_L-\eps)]$.
In order to emphasize the interval in which $\phi_i'(0)$ lies, we write $\pi_i^* (A_L-\eps):=\phi_i' (0)$.  
By definition of relaxed solutions, we have 
\begin{eqnarray} 
\label{eq:2}
  \text{ either } && L(\lambda , \pi_1^*(A_L-\eps),\dots,\pi_N^*(A_L-\eps)) \le 0 \\
\label{eq:3}    \text{ or }  && - \lambda + (A_L-\eps) \le 0
\end{eqnarray}
with $\lambda = -\partial_t \phi(t_0,x_0)$. 

We claim that \eqref{eq:3} always holds true.  We argue by
contradiction by assuming that \( (A_L-\eps) > \lambda \).
In particular $A_L > \lambda$ and $\pi_i^+(A_L) > \pi_i^*(A_L-\eps)$
for $i=1,\dots, N$.  Using (L3) and \eqref{eq:2} successively, we have
\[ L (A_L, \pi_1^+(A_L),\dots,\pi_N^+(A_L)) < L (\lambda,
\pi_1^*(A_L-\eps),\dots,\pi_N^*(A_L-\eps)) \le 0 \]
which contradicts \eqref{view:def}. The reader may remark that the
contradiction cannot be reached without the use of $\eps$.

We now consider
\[
 L_{A_L-\eps}(-\partial_t \varphi,\partial_1 \varphi, \dots, \partial_N \varphi,x'_0,D'\varphi) 
=  
\partial_t \phi (t_0,x_0') + \max (A_L-\eps,\max_i H_i^- (\pi_i^*(A_L-\eps))). 
\]
where the derivatives of $\varphi$ in the left hand side are computed
at $(t_0,x_0)$. 

Remark now that $H_i^- (\pi_i^*(A_L-\eps))= \min_{p_i \in \R} H_i(p_i)$ and in
particular $\max_i H_i^- (\pi_i^*(A_L-\eps)) = A_0$. 
Since $A_L -\eps > A_0$ and $\lambda =  -\partial_t \phi(t_0,x_0)$, 
the previous equality and \eqref{eq:3} yield
\[
 L_{A_L-\eps}(-\partial_t \varphi,\partial_1 \varphi, \dots, \partial_N \varphi,x'_0,D'\varphi)  
=  -\lambda + A_L(x'_0,p'_0) -\eps\le 0
\]
which is the desired inequality. The proof is now complete. 
\end{proof}

\section{The vanishing viscosity limit}
\label{sec:vvl}

This section is devoted to the study of the limit (as $\eps \to 0$) of the solution
$u^\eps$ of the following Hamilton-Jacobi equation posed on a multi-dimensional junction $J$, 
\begin{equation}\label{eq:dp-visc}
 \left\{
\begin{array}{ll}
u^\eps_t +  H_i(t,x,Du^\eps) = \eps \Delta u^\eps & (t,x)\in (0,T) \times J_i^*, \medskip \\
L(-u^\eps_t,\partial_1 u^\eps,\dots,\partial_N u^\eps,t,x',D'u^\eps)=0 & (t,x) \in (0,T)\times \Gamma
\end{array}
\right.
\end{equation}
subject to the initial condition
\begin{equation}\label{eq:ic-bis}
u(0,x) = u_0 (x), \quad x \in J.
\end{equation}
Notice that this equation is not of the form~\eqref{eq:dp-gen} since
the diffusion does not degenerate along the junction hyperplane. In
particular, Theorem~\ref{thm:reduced} does not hold true anymore in
this case since it uses the degeneracy along $\Gamma$ in an essential
way. Still, we can consider relaxed solutions as in
Definition~\ref{defi:relaxed}, even if we expect solutions to be
classical -- see Remark~\ref{rem:c2} below. As we shall see it, the
solutions $u^\eps$ converge towards the solution of
\begin{equation}\label{eq:hj}
 \left\{
\begin{array}{ll}
u_t +  H_i(t,x,Du) = 0 & (t,x)\in (0,T) \times J_i^*, \medskip \\
L(-u_t,\partial_1 u,\dots,\partial_N u,t,x',D'u)=0 & (t,x) \in (0,T)\times\Gamma.
\end{array}
\right.
\end{equation}

 The first result applies to general junction functions $L$.
\begin{thm}[Vanishing viscosity limit]\label{thm:vvl}
  Assume  \emph{(L)}  and 
  \[ \begin{cases}
    H_i \text{ continuous }\\
    \forall (t,x') \in (0,T) \times \Gamma, \lambda \in \R, \;\;\{ p=(p',p_i) \in \R^{d+1}: H_i (t,x',p) \le \lambda \} \text{ convex  }  \\
    \lim_{|p|\to + \infty} \inf_{(t,x') \in (0,T) \times \Gamma} H_i
    (t,x',p) = + \infty.
\end{cases}
\]
Let $u_0$ be uniformly continuous in $J$.  Assume there exists a
relaxed solution $u^\eps$ of \eqref{eq:dp-visc}, \eqref{eq:ic-bis} and
a constant $C$ such that $|u^\eps(t,x)-u_0(x)| \le Ct$ for all
$(t,x) \in (0,T) \times J$. Then $u^\eps$ converges locally uniformly
towards the unique relaxed solution $u$ of \eqref{eq:hj},
\eqref{eq:ic-bis}.
\end{thm}
\begin{rem} \label{rem:c2}
  Even if we will not discuss it, the existence of solutions whose
  restriction to $J_i$ are $C^{1,1} (J_i) \cap C^2 (J_i^*)$ is
  expected in the case of \eqref{eq:dp-visc}. Some results are proved
  in \cite{b1,b2} on compact junctions and some others are announced
  in \cite{ls}.
\end{rem}
\begin{rem}
  As we previously mentioned it, a special case of the theorem is
  proved in \cite{cms13}.
\end{rem}
\begin{proof}[Proof of Theorem~\ref{thm:vvl}]
  By discontinuous stability, the relaxed upper limit $\bar u$ of
  $u^\eps$ is an $L$-relaxed sub-solution of \eqref{eq:hj}, i.e. an
  $A_L$-flux-limited sub-solution of \eqref{eq:hj} (by
  Theorem~\ref{thm:class}). The relaxed lower limit $\underline u$ is
  an $L$-relaxed super-solution of \eqref{eq:hj}, i.e. an
  $A_L$-flux-limited super-solution of \eqref{eq:hj} (by
  Theorem~\ref{thm:class} again). Moreover, the fact that
  $|u^\eps(t,x)-u_0(x)| \le Ct$ holds true for all
  $(t,x) \in (0,T) \times J$ implies that
  $\bar u (0,x) = u_0 (x) = \underline{u} (0,x)$ for all $x \in J$. By
  comparison principle \cite[Theorem~1.3]{im14}, we conclude that
  $\bar u \le \underline{u}$ which yields the local uniform
  convergence towards the unique $A_L$-flux-limited solution of
  \eqref{eq:hj}, \eqref{eq:ic-bis} which coincides with the relaxed
  solution (by Theorem~\ref{thm:class}).
\end{proof}

Problem~\eqref{eq:vvl} can be translated 
 into the junction framework as follows, 
\[ 
\begin{cases}
u^\eps_t + H_i (t,x,D u^\eps)=\eps \Delta u^\eps, & (t,x) \in (0,T) \times J_i^*, \quad i=1,2 \\
- \partial_1 u^\eps (t,x',0)- \partial_2 u^\eps(t,x',0) = 0, & (t,x') \in (0,T) \times \Gamma \\
u^\eps(0,x) = u_0 (x), & x \in J
\end{cases}
\]
with $H_1(x,p',p_{d+1}) = \tilde{H}_1(x,p',-p_{d+1})$ and
$H_2(x,p',p_{d+1}) = \tilde{H}_2 (x,p',p_{d+1})$.  
In view of Theorem~\ref{thm:vvl}, $u^\eps$ converges towards the relaxed solution
\begin{equation} \label{eq:hj-lim}
 \begin{cases} u_t + H_i (t,x,Du) = 0, &  (t,x) \in (0,T) \times J_i^* \\ 
u(0,x) = u_0 (x), & x \in J \end{cases} 
\end{equation}
associated with the generalized flux function 
\[ L_e (p_0,p_1,p_2,t,x',p') = -p_1 -p_2.\]
\begin{cor}[The vanishing viscosity limit for the Kirchoff condition]\label{cor:ae}
  The solution $u^\eps$ of \eqref{eq:dp-visc}, \eqref{eq:ic-bis}
  converges towards the $A_e$-flux-limited solution of
  \eqref{eq:hj-lim} where $A_e(t,x',p')$ is determined as follows: if
  $p_1^0 (t,x',p')+ p_2^0(t,x',p') \ge 0$ then
  $A_e(t,x',p')=A_0(t,x',p')$; else $A_e (t,x',p')$ is the unique
  $\lambda \ge A_0(t,x',p')$ such that there exist
  $p_1^{+,e} (t,x',p')\ge p_1^0(t,x',p')$ and
  $p_2^{+,e} (t,x',p') \ge p_2^0(t,x',p')$ such that
\[ H_i (t,x',p',p_i^{+,e}(t,x',p')) = \lambda \text{ for } i=1,2, \quad p_1^{+,e} (t,x',p')+ p_2^{+,e}(t,x',p')=0.\]
\end{cor}
\begin{rem}
  If $H_1$ and $H_2$ has no constant parts and
  $p_1^0 (t,x',p')+ p_2^0(t,x',p') \le 0$, then $A_e(t,x',p')$ is the only
  $A \in \R$ such that \(\pi_1^+ (t,x',p',A) + \pi_2^+(t,x',p',A) =0\).
\end{rem}

We now recall the result about maximal and minimal Ishii solutions from \cite[Proposition~4.1]{im14}. 
\begin{prop}[Maximal and minimal Ishii solutions -- {\cite[Proposition~4.1]{im14}}] \label{prop:ishii} 
The maximal (respectively the  minimal) Ishii solution of \eqref{eq:hjtilde} corresponds to the
  $A_I^-$ (respectively $A_I^+$) flux-limited solution of
  \eqref{eq:hj-lim} with
\[ \begin{cases}
 A_I^+ (t,x',p') = \max (A_0 (t,x',p'), A^* (t,x',p')) \\
A_I^- (t,x',p') = \begin{cases} A_I^+ (t,x',p') & \text{ if } \;\; \pi^0_2 (t,x',p') + \pi^0_1 (t,x',p') \le 0\\
A_0 (t,x',p') & \text{ if } \;\;  \pi^0_2 (t,x',p') +\pi^0_1 (t,x',p') \ge 0\end{cases}
\end{cases}\]
where 
\[ A^*(t,x',p') = \max_{p_{d+1} \in I(t,x',p')} \big( \min (H_2 (t,x',p',p_{d+1}), H_1 (t,x',p',-p_{d+1}) \big)  \]
and $I(t,x',p')=[\min(-\pi^0_1 (t,x',p'), \pi_2^0 (t,x',p')), \max(-\pi^0_1 (t,x',p'),\pi_2^0 (t,x',p')) ]$.  
\end{prop}
\begin{rem}
The functions $p_i^0$ and $\pi_i^0$ are different. The Hamiltonian $H_i$ achieves its minimum at $\pi_i^0$ and 
it reaches the value $A_0$ at $p_i^0$. The only case where these functions coincide is when $A_0 = \min_{p_i} H_i (p_i)$ 
but in general $A_0 \ge  \min_{p_i} H_i (p_i)$.
\end{rem}
We now prove the following theorem, which is equivalent to Theorem~\ref{thm:vvl-evol-tilde}.
\begin{thm}[The vanishing viscosity limit selects the maximal Ishii solution]\label{thm:vvl-evol}
Assume
\[ \begin{cases}
H_i \text{ continuous }\\
\{ p \in \R^{d+1} : H_i (t,x',0,p) \le \lambda \} \text{ convex for all } \lambda \in \R, \\
\lim_{|p|\to + \infty} \inf_{(t,x') \in (0,T) \times \Gamma} H_i (t,x',p) = + \infty. 
\end{cases}
\]
Then the relaxed solution $u^\eps$ of \eqref{eq:dp-visc}, \eqref{eq:ic-bis} converges
towards the unique $A_I^-$-flux-limited solution of
\[ \begin{cases} u_t + H_i (x,Du) = 0, &  x \in J_i^* \\ u(0,x) = u_0 (x), & x \in J. \end{cases} \]
\end{thm}
\begin{proof}
  Once again, the tangential variables $(t,x',p')$ are not shown in
  order to clarify the presentation.

In view of Corollary~\ref{cor:ae}, we only have to prove that $A_e=A_I^-$
where $A_I^-$ is given by Proposition~\ref{prop:ishii}. 

If $\pi_1^0 + \pi_2^0 \ge 0$, then we know on the one hand from
Proposition~\ref{prop:ishii} that $A_I^-=A_0$ and on the other hand,
since $p_1^0 + p_2^0 \ge \pi_1^0 + \pi_2^0 \ge 0 $, we know from
Corollary~\ref{cor:ae} that $A_e = A_0$. We thus conclude that
$A_e = A_0=A_I^-$ in this case.

We now assume that $\pi_1^0 + \pi_2^0 \le 0$. In particular,
Proposition~\ref{prop:ishii} implies that
\begin{equation}\label{eq:ai-}
 A_I^- = A_I^+ = \max (A_0 ,A^*) 
\end{equation}
with 
\[ A^* = \max_{q \in [\pi_2^0,-\pi_1^0]}  \min (H_1(-q),H_2(q)). \]

Remark that the function $H_2$ is non-decreasing on the interval
$[\pi_2^0,-\pi_1^0]$ and the function $\tilde{H}_1 (q)=H_1 (-q)$ is
non-increasing.  We are going to distinguish three cases
as shown in Figure~\ref{fig:vvl}. Either the graphs of $H_2$ and
$\tilde{H}_1$ do not intersect on the interval $[\pi_2^0,-\pi_1^0]$ and $H_2$ 
is above (Case~1), or they do intersect (Case~2), or they do
not intersect and $\tilde{H}_1$ is above (Case~3). To distinguish cases,
it is enough to compare the  values of $\tilde{H}_1$ and $H_2$ at the
boundary of the interval. 
\begin{figure}
\includegraphics[height=6cm]{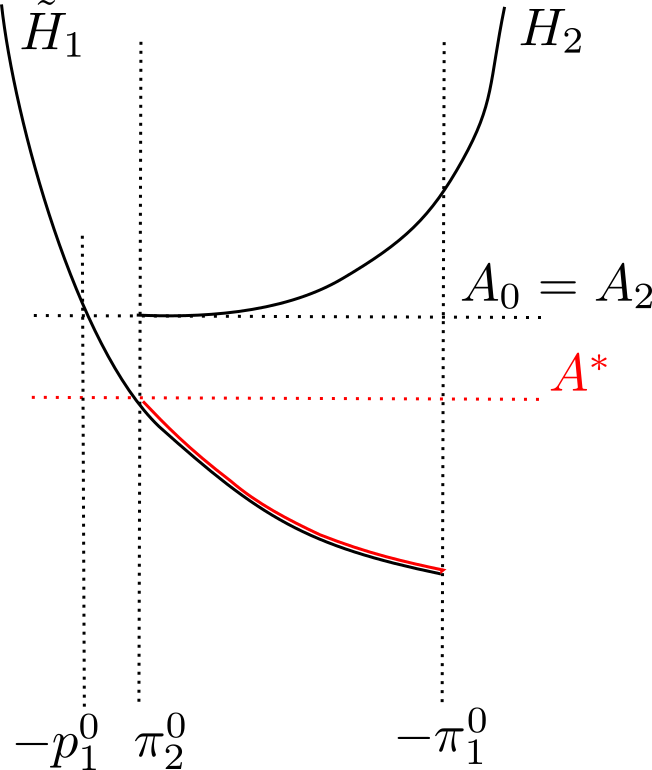} \hspace{2mm}
\includegraphics[height=6cm]{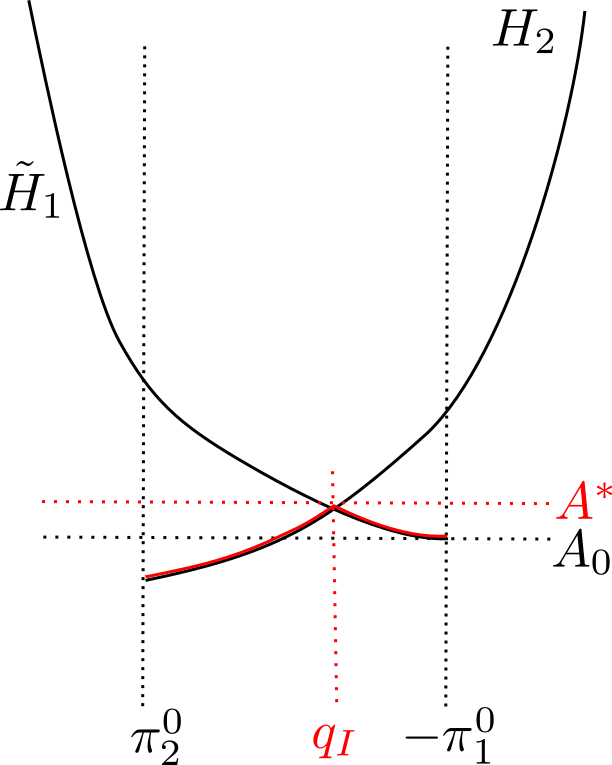} \hspace{2mm}
\includegraphics[height=6cm]{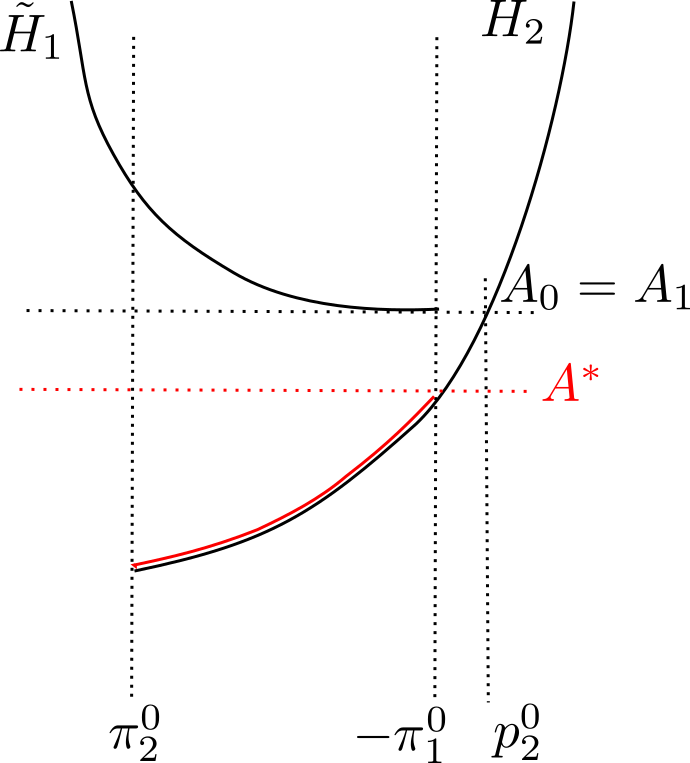}
\caption{Three cases: Case 1 (left), Case 2 (center), Case 3 (right)}
\label{fig:vvl}
\end{figure}

It is useful to introduce $A_1 = \min_{p_1 \in \R} H_1(p_1)$ and
$A_2 = \min_{p_2 \in \R} H_2(p_2)$.  Recall that $A_0 = \max(A_1,A_2)$.  

In Case~1, we have $H_2(\pi_2^0)=A_0 =A_2 \ge \tilde{H}_1(\pi_2^0)$.
It implies that $\tilde{H}_1 \le H_2$ on the interval. In particular
$A^* = \tilde{H}_1 (\pi_2^0) \le A_0$.  On the one hand,
\eqref{eq:ai-} implies that $A_I^- = A_0$.  On the other hand, since
$A_0=A_2$, we have $p_2^0 := \pi_2^+ (A_0) = \pi_2^+ (A_2)$ and
$p_1^0 := \pi_1^+(A_0) \ge -p_2^0$ (have a look at the picture).  We
thus conclude from Corollary~\ref{cor:ae} that $A_e = A_0$.  Hence
$A_I^- = A_e$ in Case~1.

In Case~2, there exists $q_I \in [\pi_2^0,-\pi_1^0]$ such that
$A^* = H_2 (q_I) = H_1 (-q_I)$ and $A^* \ge A_0$.  \eqref{eq:ai-}
implies that $A_I^- = A^*$.  But the fact that $q_I \ge \pi_2^0$ such
that $A^* = H_2 (q_I)$ implies that $q_I = \pi_2^+ (A^*)$; similarly,
$-q_I = \pi_1^+(A^*)$; hence $\pi_1^+ (A^*) + \pi_2^+(A^*) = 0$ with
$A^* \ge A_0$. We thus have from Corollary~\ref{cor:ae} that
$A_e = A^*$.  Hence $A_I^- = A_e$ in Case~2.

In Case~3, $A_0 = A_1 \ge A^*$. \eqref{eq:ai-} implies that
$A_I^- = A_0$. We also remark that
$-p_1^0 = - \pi_1^0 \le \pi_2^+ (A_0) = p_2^0$ (have a look at the
picture).  In particular, we have from Corollary~\ref{cor:ae} that
$A_e = A_0$. We thus conclude that $A_I^- = A_e$ in Case~3.

The proof is now complete. 
\end{proof}

\section{A large deviation problem}
\label{sec:ld}

In \cite{bde}, the authors study large deviation problems related to
diffusion processes whose drift is smooth on either side of a
hyperplane. Their proofs rely on probability tools and ideas.  Our
goal in this section is to propose an analytical/PDE
proof. Furthermore, by using the results of previous sections, the
rate function is related to the maximal Ishii solution of a
Hamilton-Jacobi equation.
\medskip

Consider the stochastic differential equation in $\R^{d+1}$, 
\begin{eqnarray}\label{SDE}
dX^\e(t)=b(X^\e(t))dt+\e^{1/2}\sigma(X^\e(t))dW(t), \quad X^\e(0)=x_0,\quad 0 \le t \le 1
\end{eqnarray}
with 
\[ b (x) = \begin{cases} b_1 (x) & \text{ if } x_{d+1} < 0 \\ 
b_2 (x) & \text{ if } x_{d+1} > 0\end{cases} \]
and 
\[ \sigma (x) = \begin{cases} \sigma_1 (x) & \text{ if } x_{d+1} < 0 \\ 
\sigma_2 (x) & \text{ if } x_{d+1} > 0\end{cases} \]

In order to introduce the rate function, we have to define first 
Hamiltonians and Lagrangians. Hamiltonians are defined in \cite{bde} by
\[ 
\tilde{H}_i(x,p)=\frac{1}{2} \langle a_i (x) p ,p \rangle-b_i(x) p, \quad x,p \in \mathbb{R}^{d+1}
\]
with $a_i = \sigma_i \sigma_i^T$. 
  Corresponding
Lagrangians $\tilde{L}_1$ and $\tilde{L}_2$ are related to Hamiltonians $\tilde{H}_1$ and $\tilde{H}_2$  by
the following formula \cite{bde}
\[
\tilde{H}_i(x,p)=\sup_{q \in  {\R^{d+1}}}\{-pq-\tilde{L}_i(x,q)\}.
\]
Set $\Omega_1=\R^d \times (-\infty,0),\Omega_2=\R^d \times (0,+\infty)$, $\mathcal{H}=\R^d \times \{0\}$.
\begin{equation}
\label{Lagra_juntion}
\tilde{L}(x,p)=
 \left\{
\begin{array}{ll}
\tilde{L}_1(x,p), & x \in \Omega_1, \medskip\\
\tilde{L}_2(x,p), & x \in \Omega_2, \medskip \\
\tilde{L}_0(x,p), & x \in \mathcal{H},
\end{array}
\right.
\end{equation}
where $\tilde{L}_0$ is defined by
\[
\tilde{L}_0(x,p',q)=\inf\left\{\lambda \tilde{L}_1(x,p',q_1)+(1-\lambda)\tilde{L}_2(x,p',q_2), \begin{cases} \lambda \in [0,1], q_1 \ge 0, q_2 \le 0, \\
\lambda q_1+(1-\lambda)q_2=q \end{cases} \right\}.
\]

Call $\Sigma_{x_0}$ the set of all absolutely continuous function $\phi \in C([0,1],\R^{d+1})$ satisfying $\phi(0)=x_0$.  For any $\phi \in \Sigma_{x_0}$, we define the rate function $I_{x_0}(\phi)$  as follows,
\begin{equation}\label{def:I}
I_{x_0}(\phi)=\int_0^1 \tilde{L}(\phi(s),\dot{\phi}(s))\, ds
\end{equation}
where $\tilde{L}$ is defined as in~\eqref{Lagra_juntion}.  We first state the
Laplace principle as presented
in~\cite{bde}
\begin{defn}\label{LP}
  Let $\{Y^\e(t),\e>0, 0 \le t \le 1\}$ with $Y^\e(0)=x_0$ be a family of random variables taking values in
  a Polish space $\mathcal{Y}$ and let $I_{x_0}$ be a rate function defined as in \eqref{def:I}. We say
  that $\{Y^\e\}$ satisfies a Laplace principle with the rate function $I_{x_0}$
  if, for every bounded continuous function $h$ mapping $\mathcal{Y}$
  into $\R$,   we have
\begin{eqnarray}\label{LP rate I}
\lim_{\e \to 0}\e \ln \mathbb{E}_{x_0}\big\{\exp\big[-\frac{h(Y^\e)}{\e}\big]\big\}=-\inf_{\phi  \in \Sigma_{x_0}}\{h(\phi(1))+I_{x_0}(\phi)\}.
\end{eqnarray}
\end{defn}
In \cite{bde}, the following large deviation result is proved
using probabilistic arguments. We will give a PDE proof.
\begin{thm}[\cite{bde}]\label{LDP}
Assume that
\begin{eqnarray*}
&& \left\{
\begin{array}{l}
  \text{ $b_i$ is continuous},
  \\[1mm]
  \text{ $\sigma$ is continuous and such that
$\sigma \sigma^T \ge c \mathcal{I}$ with $c>0$},\\[1mm]
  \text{ \eqref{SDE} has a unique strong solution}
\end{array}
\right.
\end{eqnarray*}
where $\mathcal{I}$ is the identity matrix. 
Then the family $\{X^\e,\e>0\}$ satisfies the Laplace principle in $C([0,1],\R^{d+1})$ with 
the rate function $I_{x_0}$ as defined in \eqref{def:I}.
\end{thm}
\begin{proof}
Given a function $h$, let $h_\e$ denote $\exp(\frac{-h}{\e})$. The function $u_\e$ given by
\[
u_\e(t,x)=\mathbb{E}_x (h_\e(X^\e(t)))
\]
is a solution of
\[
\begin{cases}
 \frac{\partial u_\e}{\partial t}=\frac{\e}{2} \trace (a (x) D^2u_\e) +b(x) Du_\e, &  t \in (0,1), x \in \Omega_1 \bigcup \Omega_2
\\[1mm]
\frac{1}{2}\partial_{d+1} u_\e (t,x',0^+)=\frac{1}{2}\partial_{d+1} u_\e (t,x',0^-), & x \in \mathcal{H}
\\[1mm]
u_\e(0,x)=h_\e(x), & x \in \Omega_1 \bigcup \Omega_2
\end{cases}
\]
(where $a = \sigma \sigma^T$)
The function
\(
v_\e=-\e \ln (u_\e)
\)
satisfies
\[
\begin{cases}
\frac{\partial v_\e}{\partial t}=\frac{\e}{2} \trace (a(x)  D^2v_\e)-\frac{1}{2}  \langle a(x) Dv_\e,Dv_\e\rangle+b(x) Dv_\e, & t \in (0,1), x \in \Omega_1 \bigcup \Omega_2
\\[1mm]
\frac{1}{2}\partial_{d+1} v_\e (t,x',0^+)=\frac{1}{2}\partial_{d+1} v_\e (t,x',0^è), & x \in \mathcal{H}
\\[1mm]
v_\e(0,x)=h(x), & x \in \Omega_1 \bigcup \Omega_2.
\end{cases}
\]
Moreover, in view of the definition of $u_\e$ and $v_\e$, we have
\begin{eqnarray*}
v_\e(t,x)=-\e \ln \mathbb{E}_x \left \{ \exp \left[ \frac{-h(X^\e(t))}{\e} \right] \right\}.
\end{eqnarray*}
Hence, our goal is to prove that 
\begin{eqnarray*}
\lim_{\e \to 0} v_\e (1,x)=\inf_{\phi \in \Sigma_x}\{h(\phi(1))+I_{x}(\phi)\}
\end{eqnarray*}
where $I_x$ is defined in \eqref{def:I}.

We know from Theorem~\ref{thm:vvl-evol-tilde} that $v_\e$ converges locally uniformly towards the maximal Ishii solution $U^+$ of 
\begin{equation}\label{HJ}
\begin{cases}
 \frac{\partial U^+}{\partial t}+\tilde{H}_i(x,DU^+)=0, & x \in \Omega_i, \quad t \in (0,1) \\[1mm]
U^+(0,x)=h(x), & x \in \Omega_1 \bigcup \Omega_2 .
\end{cases}
\end{equation}
It thus remains to prove that 
\begin{equation}\label{u+formula}
 U^+(1,x) = \inf_{\phi \in \Sigma_x}\{h(\phi(1))+I_{x}(\phi)\}.
\end{equation}
In view of the definition of Lagrangians and Hamiltonians from \cite{bde} recalled above, we have 
\[ \tilde{H}_i (x,p) =\sup_{q \in  \R^{d+1}}\{pq-l_i(x,q)\} \quad \text{ with } \quad  l_i(x,-q)=\tilde{L}_i(x,q),
\]
here $l_i$ corresponds to the running costs considered in~\cite[Section~6]{im13}.
In view of the definition of $\tilde{L}_0$ recalled above, we have 
\begin{align*}
\tilde{L}_0(x,q',0)&=\inf \left\{\lambda \tilde{L}_1(x,q',q_1)+(1-\lambda)\tilde{L}_2(x,q',q_2), \begin{cases} 0 \le \lambda \le 1, \\ q_1 \ge 0, q_2 \le 0,\lambda q_1+(1-\lambda)q_2=0 \end{cases} \right\}\\
&=\inf \left\{\lambda l_1(x,q',v_1)+(1-\lambda)l_2(x,q',v_2), \begin{cases} 0\le \lambda \le 1, \\v_1 \le 0,  v_2 \ge 0,\lambda v_1+(1-\lambda)v_2=0 \end{cases} \right \}.
\end{align*}
Hence, the formula of $U^+$ given in~\cite{im14,bbcim} coincides with \eqref{u+formula}. The proof is now complete. 
\end{proof}

\noindent \textbf{Acknowledgement.} The authors thank Russell Schwab
for fruitful discussions in early stage of this work and for
suggesting to address the large deviation problem. They also thank Guy
Barles for stimulating discussions about the vanishing viscosity
limit. The authors are also indebted to one of the referrees who read
very carefully two successive versions of this work and made valuable
recommandations about both presentation and proofs.
\bibliographystyle{plain} \bibliography{vvl}

\end{document}